\documentclass{amsart}

\calclayout

\newtheorem{theorem}{Theorem}[section]
\newtheorem{lemma}[theorem]{Lemma}
\newtheorem{proposition}[theorem]{Proposition}

\theoremstyle{definition}
\newtheorem{definition}[theorem]{Definition}

\newtheorem{assumption}[theorem]{Assumption}
\newtheorem{construction}[theorem]{Construction}

\theoremstyle{remark}
\newtheorem{remark}[theorem]{Remark}

\makeatletter
\@namedef{subjclassname@2020}{\textup{2020} Mathematics Subject Classification}
\makeatother

\numberwithin{equation}{section}

\newcommand{\diag}{\operatorname{diag}}
\newcommand\HH{\mathrm{H}}
\newcommand{\Pic}{\operatorname{Pic}}
\newcommand{\ad}{\operatorname{ad}}

\newcommand{\pr}{\operatorname{pr}}
\newcommand{\Nm}{\operatorname{Nm}}
\newcommand{\Prym}{\operatorname{Prym}}

\newcommand{\rank}{\operatorname{rank}}
\newcommand{\Hom}{\operatorname{Hom}}

\newcommand{\Sym}{\operatorname{Sym}}

\newcommand{\ord}{\operatorname{ord}}
\newcommand{\Tr}{\operatorname{Tr}}
\newcommand{\Jac}{\operatorname{Jac}}

\newcommand{\CC}{\mathbb{C}}
\newcommand{\PP}{\mathbb{P}}
\newcommand{\ZZ}{\mathbb{Z}}

\newcommand{\id}{\operatorname{id}}

\newcommand{\Ann}{\operatorname{Ann}}

\let\varphi\phi
\let\bar\overline
\let\tilde\widetilde
\let\hat\widehat

\usepackage{xcolor}
\usepackage{cite}

\DeclareSymbolFont{script}{U}{eus}{m}{n}
\DeclareMathSymbol{\Wedge}{0}{script}{"5E}

\newcommand{\GL}{\mathrm{GL}}

\newcommand{\Sp}{\mathrm{Sp}}
\newcommand{\SO}{\mathrm{SO}}

\begin{document}

\title{Spectral Data of Special Orthogonal Higgs bundles and Hecke Modification}

\author{Tyson Klingner}
\address{University of Washington, Department of Mathematics, Seattle, WA 98195, USA}
\curraddr{}
\email{tysonk4@uw.edu}
\thanks{}



\begin{abstract}
We give a complete, self-contained computation of the spectral data parametrising Higgs bundles in the generic fibres of the $\SO_{2n+1}$-Hitchin fibration where the Higgs fields are $L$-twisted endomorphisms. Although the spectral data is known in the  literature, we develop a new approach, which takes advantage of Hecke modification. Further, we present the computation for the $\Sp_{2n}$ and $\SO_{2n}$ cases while clarifying some aspects of the correspondence which are not well explained in the pre-existing literature. We also compute the number of connected components of the generic fibres, and demonstrate Langlands duality in the fibres via the canonical duality in the fibres.
\end{abstract}

\maketitle

\section{Introduction}\label{sec1}

Let $X$ be a compact Riemann surface of genus at least $2$. The goal of this paper is to completely describe the generic fibres of the $\SO_{m}$-Hitchin fibration 
\[
h : \mathcal{M}({\SO_{m}, L}) \to B := 
\begin{cases}
\bigoplus_{i=1}^{n-1} \HH^0(X, L^{2i}) \oplus \HH^0(X, L^n)  & \text{if $m=2n$} \\
\bigoplus_{i=1}^{n} \HH^0(X, L^{2i}) & \text{if $m=2n+1$.}
\end{cases}
\] 
Here, $\mathcal{M}({\SO_{m}, L})$ denotes the moduli space of $L$-twisted $\SO_{m}$-Higgs bundles over $X$ \cite[Theorem 5.10]{MR1085642}, and $\pi : L \to X$ is a non-trivial basepoint-free holomorphic line bundle. The Hitchin fibration defines a proper, surjective holomorphic mapping \cite[Theorem 6.1]{MR1085642} (cf. \cite{MR0887284, MR0885778} for $L=K_X$). Hitchin demonstrated in \cite[Sec. 5.17]{MR0885778} that when $L=K_X$ the Hitchin fibration endows $\mathcal{M}(\SO_{m}, K_X)$ with the structure of  an algebraically complete integrable system with respect to the canonical symplectic structure. Consequently, the generic fibres are torsors over abelian varieties. For $L \ne K_X$ the moduli space $\mathcal{M}(\SO_{2n+1}, L)$ does not have a symplectic structure and we cannot view the space as an integrable system. However, we show by direct computation that the generic fibres of $h : \mathcal{M}(\SO_{m}, L) \to B$ are still torsors of abelian varieties, which we describe by {\em spectral data.}

\bigskip

More specifically, a generic point $a \in B$ naturally defines a complex analytic subvariety $S \subset \operatorname{Tot}(L)$ called a {\em spectral curve} (see Section \ref{sec2}). When $m = 2n+1$ the spectral curve $S$ associated to $a$ is smooth (see Sections \ref{secsp2n} and \ref{secso2n+1}). However, when $m=2n$, the spectral curve $S$ is singular whose singularities are ordinary double points (see Section \ref{sec5}). In either case there is a canonical involution $\sigma : S \to S$ and for $m$ odd $\bar{S} := S/\sigma$ defines a compact Riemann surface. For the even case we consider the normalisation $\hat{S}$ of $S$ and $C := \hat{S}/\hat{\sigma}$ where $\hat{\sigma}$ is a lift of the involution.  

\bigskip

In \cite[Theorem 1]{MR2354922} and \cite[Sec. 5.14]{ MR0885778}  for $L = K_X$ Hitchin computed via spectral data that each connected component of $h^{-1}(a)$ is a torsor of the Prym varieties $\Prym(S, \bar{S})/p^*\Jac(\bar{S})[2]$ for $m=2n+1$, and $\Prym(\hat{S}, C)$ for $m = 2n$ respectively. Here, $p^*\Jac(\bar{S})[2]$ denotes the canonical action by 2-torsion bundles where $p : S \to \bar{S}$ denotes projection. See Section \ref{prymvarieties} for more about Prym varieties. Hitchin's arguments used in the odd case do not generalise to the case $L \ne K_X$. However, the analogous result does indeed hold.  

\begin{theorem}[Spectral Data] \label{main theorem} Let $a \in B$ be a generic point. Then the fibre of the $\SO_{m}$-Hitchin fibration $h^{-1}(a)$ consists of two connected components each of which defines a torsor of the Prym variety:
\begin{equation}
 \begin{cases}
\Prym(S, \bar{S})/p^*\Jac(\bar{S})[2] & \text{if $m = 2n+1$} \\
\Prym(S, \bar{S}) & \text{if $m=2n$}.

\end{cases}
\end{equation}

\end{theorem}

To compute the spectral data, we give a new approach in the $m=2n+1$ case in Section \ref{secso2n+1}, which takes advantage of {\em Hecke modifications} of orthogonal bundles (c.f. \cite{MR3049307}).  Suppose $(V, \phi, Q)$ is a generic $\SO_{2n+1}$-Higgs bundle. The associated spectral curve is always singular. Away from the singularities we compute the spectral data following Hitchin's method in reducing to a smooth $\Sp_{2n}$ spectral curve $(S, \sigma)$ associated to a $\Sp_{2n}$-Higgs bundle $(E, \Phi, \omega)$. However, contrary to Hitchin we reduce to a $L$-valued symplectic form (Construction \ref{firstone}). The spectral data associated to the $\Sp_{2n}$-Higgs bundle is a holomorphic line bundle $N \in \Jac(\bar{S})$ along with an isomorphism $\theta : \sigma(N) \to N^* \otimes \pi^*(L^{2n})$, see Section \ref{SO_{2n+1} spec data}. To give the spectral data for the $\SO_{2n+1}$-Higgs bundle we reconstruct $(V, \phi, Q)$ from $(E, \Phi, \omega)$ by identifying $\mathcal{O}_X(V)$ to a subsheaf of  $\mathcal{O}_X(L^{-1}E \oplus L^{-n}) \otimes_{\mathcal{O}_X} \mathcal{M}_X$, and applying a Hecke transformation at the image of each singularity to recover the Higgs field $\Phi$ and the orthogonal structure $Q$. The $\Sp_{2n}$ spectral data $(N, \theta)$ and Hecke transformations completely describe the spectral data. This use of Hecke transformations provides the full generalisation to the arbitrary $L$ case. 

\bigskip

In Sections \ref{secsp2n} and \ref{sec5} we compute the spectral data of the $\Sp_{2n}$ and $\SO_{2n}$ cases respectively. Although the spectral data is known in each case where Hitchin's arguments do generalise to the arbitrary $L$ case, we clarify some aspects of the correspondence, which are not well explained in the pre-existing literature. For the example, the $\SO_{2n}$ computation relies on the fact that the BNR-correspondence (Theorem \ref{BNR}) extends to the normalisation $\hat{S}$ of the singular $\SO_{2n}$ spectral curve $S$. In Section \ref{SO_{2n} BNR-correspondence} we give a complete treatment using the trace map to show that this is the case. Moreover, in Section \ref{det E is trivial} we give the full computation that the $\mathrm{O}_{2n}$-Higgs bundle constructued from the spectral data does have trivial determinant along with a unique orientation making the $\mathrm{O}_{2n}$-Higgs bundle a $\SO_{2n}$-Higgs bundle establishing the correspondence.   

\section*{Acknowledgments}

The author thanks his Masters supervisor Dr. David Baraglia at the University of Adelaide for unwavering support and help. The author is thankful for the referee's careful reading and feedback. 

\section{Higgs Bundles and Spectral Curves}\label{sec2}

Throughout the article $X$ denotes a compact Riemann surface with genus $g \ge 2$ and $\pi: L\to X$ denotes a fixed basepoint-free holomorphic line bundle with positive degree. We begin by recalling the definitions of Higgs bundles and spectral curves which lay the foundations for the results. In what proceeds $G$ denotes an arbitrary complex reductive Lie group and $\mathfrak{g}$ denotes the associated Lie algebra.

\begin{definition}
A {\em $G$-Higgs bundle} is a pair $(P, \Phi)$ where $P$ is a holomorphic principal $G$-bundle over $X$ and $\Phi$ is a holomorphic section of $\ad(P) \otimes L$. Here $\ad(P)$ denotes the adjoint bundle associated to $P$. The homomorphism $\Phi$ is called the {\em Higgs field}.
\end{definition}
\begin{remark}
Rank $n$ holomorphic vector bundles are canonically associated to principal $\GL_n$-bundles via the frame bundle. Under this identification, the adjoint bundle of a principal $\GL_n$-bundle is canonically isomorphic to the endomorphism bundle of the corresponding holomorphic vector bundle. Hence, a $\GL_n$-Higgs bundle is a pair $(E, \Phi)$ where $E \to X$ is a rank $n$ holomorphic vector bundle and $\Phi : E \to L \otimes E$ is a homomorphism.
\end{remark}

Classically Higgs bundles are defined for $L = K_X$. However, the generalisation of Higgs bundles to non-compact surfaces has been established in the literature, e.g., \cite{MR1040197, MR4027558, boalch2012}. The generalised objects are {\em parabolic Higgs bundles} \cite[Def. 5.3]{MR1353317}, for which the Higgs field $\Phi$ takes values in a non-trivial basepoint-free holomorphic line bundle $L$ over $X$, which may not be $K_X$. Hence, in this paper we follow the modern convention of using arbitrary $L$. More recently, Higgs bundles have been generalised to, and studied over, higher dimensional varieties over an arbitrary algebraically closed field $k$ of characteristic zero. For more details, consult \cite{MR4686658, MR4118645, chen2018}.

\bigskip

Higgs bundles with structure group $G \subseteq \GL_n$ have become of great interest in understanding representations of the fundamental group $\pi_1(X)$ under the {\em nonabelian Hodge correspondence} established by Corlette and Simpson. For more details see \cite{MR1179076, MR1307297, MR1320603, MR0965220}. If $G \subseteq \GL_n$, a $G$-Higgs bundle naturally defines a pair $(E, \Phi)$ as per the remark. By pulling back $L$ over itself we obtain the tautological section $\lambda \in \HH^0(L, \pi^* L)$. The {\em spectral curve} associated to $(E, \Phi)$ denoted $S$ is the complex analytic subvariety of $\operatorname{Tot}(L)$ defined by the zero locus of 
\begin{equation}
\label{charac}
p(\lambda) := \det(\lambda - \pi^*(\Phi)) = \lambda^n + \pi^*(a_1) \lambda^{n-1} + \cdots + \pi^*(a_n)
\end{equation}
See \cite[Sec. 3]{MR998478} for a detailed account of the construction of a general spectral curve $S$. By abuse of notation we will omit the pullbacks when denoting the characteristic coefficients in the characteristic polynomial. 

\bigskip

If $G = \GL_n$, then the fact that $L$ is basepoint-free allows us to assume that a given spectral curve $\pi : S \to X$ is smooth, see \cite[Sec. 5.1]{MR0885778}. 
Now, by the projection formula pushing forward a line bundle $N \in \Pic(S)$ and the tautological section $\lambda : N \to \pi^*(L) \otimes N$ defines a $\GL_n$-Higgs bundle $(\pi_*(N), \pi_*(\lambda))$ whose spectral curve is $S.$ \footnote{\, Generically $\pi_*(\lambda)$ has $n$-distinct eigenvalues and $\pi_*(N)$ has rank $n$, thus the minimal polynomial and characteristic polynomial for $\pi_*(\lambda)$ agree. By the Cayley Hamilton theorem $p(\pi_*\lambda) = 0$, and since $p$ has degree $n$ it follows that $p(\lambda)$ is the characteristic polynomial for $\pi_*\lambda.$} Hence, smooth spectral curves are always associated to $\GL_n$-Higgs bundles. 

\bigskip

The trivial holomorphic line bundle $\mathcal{O}_S$ on a smooth spectral curve $S$ is of great importance to the topic. In \cite[pp. 173]{MR998478} it was shown by algebraic means that $\pi_*\mathcal{O}_S \cong \mathcal{O}_X \oplus L^{-1} \oplus \cdots \oplus L^{-(n-1)}$ where $n$ is the degree of the spectral curve, and consequently, $\pi:S \to X$ is finite. Here, we provide an alternative complex analytic proof.

\begin{proposition}
\label{structuresheaf}
The map
\[
\mathcal{O}_X \oplus L^{-1} \oplus \cdots \oplus L^{-(n-1)}  \to \pi_*\mathcal{O}_S
\]
defined by
\[
(c_0, c_1, \ldots, c_{n-1}) \mapsto \pi^*(c_0) + \pi^*(c_1)\lambda + \cdots + \pi^*(c_{n-1})\lambda^{n-1}
\]
defines an isomorphism. 
\end{proposition}

\begin{proof}
To see that $\pi^*(c_0) + \pi^*(c_1)\lambda + \cdots + \pi^*(c_{n-1})\lambda^{n-1}$ defines a section of $\pi_*\mathcal{O}_S$, suppose that $U \subseteq X$ is open. By definition, $\pi^*(c_i) \in \mathcal{O}_S(\pi^*(L^{-i}))(\pi^{-1}(U))$ and $\lambda^i \in \mathcal{O}_S(\pi^*(L^i))(\pi^{-1}(U))$, hence $\pi^*(c_i)\lambda^i \in \mathcal{O}_S(\pi^{-1}(U)) = (\pi_* \mathcal{O}_S)(U)$ for $i = 1, \ldots, n.$ Summing over the index shows that $\pi^*(c_0) + \pi^*(c_1)\lambda + \cdots + \pi^*(c_{n-1})\lambda^{n-1}$ defines a section of $\pi_*\mathcal{O}_S$.  

Fix $Y = \operatorname{Tot}(L)$. Let $\bar{Y} = \PP(\mathcal{O}_X \oplus L)$ denote the projectivisation of $L$ with projection map $\bar{\pi} : \bar{Y} \to X$, and let $\mathcal{O}_{\bar{Y}}(1)$ denote the hyperplane bundle.  The complement $Y_\infty := \bar{Y} - Y$ defines a divisor on $\bar{Y}$ whose associated line bundle is $\mathcal{O}_{\bar{Y}}(1).$ Hence, we may choose a section $s$ of $\mathcal{O}_{\bar{Y}}(1)$ whose zero locus is $Y_\infty.$ The global sections of $\mathcal{O}_{\bar{Y}}(1)$ are canonically identified to global sections of $\Hom(\mathcal{O}_X \oplus L, \mathcal{O}_X),$ and hence, we can take $s$ to be the section identified to the projection homomorphism $\pr : \mathcal{O}_X \oplus L \to \mathcal{O}_X.$ Thus, $s$ is non-vanishing on $Y$, and there is a unique trivialisation $\mathcal{O}_{\bar{Y}}(1)\vert_Y \, \cong \mathcal{O}_Y$ that sends $s\vert_Y$ to 1. Now, let $\bar{\lambda}$ be the global section of $L \otimes \mathcal{O}_{\bar{Y}}(1)$ that corresponds to $\pr : \mathcal{O}_X \oplus L \to L$. Under the previously mentioned trivialisation $\mathcal{O}_{\bar{Y}}(1)\vert_Y \, \cong \mathcal{O}_Y$ the section $\bar{\lambda}$ is sent to the tautological section $\lambda$ of $L$. Next, consider the section $\bar{p}(\bar{\lambda}) \in \HH^0(\bar{Y}, L^n \otimes \mathcal{O}_{\bar{Y}}(n))$ defined by 
\[
\bar{p}(\bar{\lambda}) = \bar{\lambda}\,^n + \pi^*(a_1)\bar{\lambda}\,^{n-1}s + \cdots + \pi^*(a_n)s^n.
\]
We claim that $S$ is the zero locus of $\bar{p}(\bar{\lambda}).$ Since $\bar{p}(\bar{\lambda}) \vert_Y \, = p(\lambda)$ under the trivialisation it suffices to show that $\bar{p}(\bar{\lambda})$ is non-zero on $Y_\infty.$ However, $s \vert_{Y_\infty} = 0$, and thus, $\bar{p}(\bar{\lambda}) \vert_{Y_\infty} \, = \bar{\lambda}\,^n \vert_{Y_\infty},$ which is non-zero. Thus, since $S$ is the zero locus of $\bar{p}(\bar{\lambda})$ there is a short exact sequence of sheaves
\begin{equation}
\label{sheavessequence}
0 \to \mathcal{O}_{\bar{Y}}(L^{-n}(-1)) \xrightarrow{\bar{p}(\bar{\lambda})} \mathcal{O}_{\bar{Y}}(n-1) \to \mathcal{O}_{S}(n-1) \to 0.
\end{equation}
From the trivialisation $\mathcal{O}_{\bar{Y}}(1)\vert_{Y} \, \cong \mathcal{O}_Y$ notice that $\mathcal{O}_S(n-1) \cong \mathcal{O}_S$, and applying the direct image functor $\bar{\pi}_*$ to (\ref{sheavessequence}), recalling $\bar{\pi} \vert_{Y} = \pi$, gives the exact sequence
\begin{equation}
\label{sheaveslongsequence}
0 \to \bar{\pi}_*\mathcal{O}_{\bar{Y}}(L^{-n}(-1)) \to \bar{\pi}_*\mathcal{O}_{\bar{Y}}(n-1) \to \pi_*\mathcal{O}_{S} \to R^1 \bar{\pi}_*\mathcal{O}_{\bar{Y}}(L^{-n}(-1))  \to \cdots
\end{equation}
By the projection formula, $\bar{\pi}_*\mathcal{O}_{\bar{Y}}(L^{-n}(-1)) \cong L^{-n} \otimes \bar{\pi}_*\mathcal{O}_{\bar{Y}}(-1) \cong 0.$ Next, $R^1\bar{\pi}_*\mathcal{O}_{\bar{Y}}(L^{-n}(-1))$ is the sheaf associated to the presheaf $U \mapsto \HH^1(\bar{\pi}^{-1}(U), \mathcal{O}_{\bar{Y}}(L^{-n}(-1)))$, and for each $x \in X$, notice that $\bar{\pi}\,^{-1}(x) \cong \PP^1$ and $L \vert_{\bar{\pi}\,^{-1}(x)} \, \cong \mathcal{O}_{\PP^1}$. Hence,
\[
\HH^1(\bar{\pi}\,^{-1}(x), \mathcal{O}_{\bar{Y}}(L^{-n}(-1))) \cong \HH^1(\PP^1, \mathcal{O}_{\PP^1}(-1)).
\]
By Serre duality $\HH^1(\PP^1, \mathcal{O}_{\PP^1}(-1)) = 0$, then by Grauert's base change theorem \cite[Theorem 8.5 (iv)]{MR0749574} we see \[R^1\pi_*\mathcal{O}_{\bar{Y}}(L^{-n}(-1)) = 0.\] Thus, the exact sequence in (\ref{sheaveslongsequence}) descends to an isomorphism $\pi_*\mathcal{O}_{\bar{Y}}(n-1) \to \pi_*\mathcal{O}_S.$ Finally, $\mathcal{O}_X \oplus L^{-1} \oplus \cdots \oplus L^{-(n-1)} \cong \pi_*\mathcal{O}_{\bar{Y}}(n-1)$ with explicit isomorphism
\begin{equation}
\label{giveniso}
(c_0, c_1, \ldots, c_{n-1}) \mapsto \pi^*(c_0)s^n + \pi^*(c_1)s^{n-1}\bar{\lambda} + \cdots + \pi^*(c_{n-1}) \bar{\lambda}\,^{n-1}.
\end{equation}
Applying the trivialisation $\mathcal{O}_{\bar{Y}}(1) \vert_Y \, \cong \mathcal{O}_Y$ then defines an isomorphism \[\mathcal{O}_X \oplus L^{-1} \oplus \cdots \oplus L^{-(n-1)} \to \pi_*\mathcal{O}_{S}\] given by
\[
(c_0, c_1, \ldots, c_{n-1}) \mapsto \pi^*(c_0) + \pi^*(c_1)\lambda + \cdots + \pi^*(c_{n-1}) \lambda^{n-1}.
\]
\end{proof}

\begin{remark}
In \cite[Sec. 3]{MR998478} Beauville, Narasimhan, and Ramanan show the natural isomorphism $\mathcal{O}_X \oplus L^{-1} \oplus \cdots \oplus L^{-(n-1)}  \cong \pi_*\mathcal{O}_S$ as follows. Choose sections $a_i \in \HH^0(X, L^i)$ such that the spectral curve $\pi : S \to X$ defined as the zero scheme of the section $\bar{p}(\bar{\lambda}) \in \HH^0(\bar{Y}, L^n \otimes \mathcal{O}_{\bar{Y}}(n))$ is integral. Then interpret the algebraic structure of $S$ to be $\operatorname{Spec}(\operatorname{\Sym}(L^{-1})/\mathcal{I})$ where $\mathcal{I}$ is the ideal sheaf generated by the image of the homomorphism $u : L^{-n} \to \operatorname{Sym}(L^{-1})$ given as the sum of the imbeddings $L^{-n} \to L^{-(n-i)}$ defined by $a_i$. Under this description the natural isomorphism is immediate. 
\end{remark}

\subsection{Relative Duality and Trace Map} Let $\pi : S \to X$ be a smooth spectral curve defined by $p(\lambda)$. For our purposes, relative duality states that for every holomorphic line bundle $N \in \Pic(S)$ there is a canonical isomorphism $\pi_*(N)^* \cong \pi_*(N^*(R)),$ where $R$ denotes the ramification divisor of $\pi.$ In \cite[pp. 6]{MR3618052}, Hitchin established the isomorphism by defining a fibrewise non-degenerate pairing. 
In fact, the proof generalises to any branched cover of compact Riemann surfaces since locally a branched cover is a spectral curve by a suitable choice of holomorphic coordinates. To establish our desired results we require the {\em trace map}. Hence, we will introduce the trace map and also provide an alternative non-degenerate pairing that involves the trace map. Let $f : \Sigma' \to \Sigma$ be a branched cover between two compact Riemann surfaces. Denote the function fields of $\Sigma$ and $\Sigma'$ by $K(\Sigma)$ and $K(\Sigma')$ respectively. Then, the {\em trace map} $\Tr_{\Sigma' / \Sigma} : K(\Sigma') \to K(\Sigma)$ is defined by 
\[
\Tr_{\Sigma'/\Sigma}(g)(x) = \sum_{f(y) = x} g(y).
\]
For each open subset $U \subseteq \Sigma$ the trace map sends meromorphic functions of $\Sigma'$ over $f^{-1}(U)$ to meromorphic functions of $\Sigma$ defined over $U$. Hence, we may regard the trace map as a sheaf map $\Tr_{\Sigma'/\Sigma} : f_* \mathcal{M}_{\Sigma'} \to \mathcal{M}_{\Sigma}$. In fact, the trace map naturally generalises to sections of holomorphic vector bundles. Indeed, if $E \to \Sigma$ is a holomorphic vector bundle, then the trace map defines a sheaf map $\Tr_{\Sigma'/\Sigma} : f_*\mathcal{M}_{\Sigma'}(f^*(E)) \to \mathcal{M}_{\Sigma}(E)$ where $\mathcal{M}_{\Sigma}(E) := \mathcal{O}_{\Sigma}(E) \otimes_{\mathcal{O}_{\Sigma}} \mathcal{M}_{\Sigma}$ denotes the meromorphic sections of $E$, and similar for $f_*\mathcal{M}_{\Sigma'}(f^*(E)).$ The reader may easily verify that $\Tr_{\Sigma'/\Sigma}$ defines a homomorphism of sheaves of $\mathcal{M}_{\Sigma}$-modules. Of course, the trace map preserves holomorphicity, and hence, the foregoing discussion remains true when restricted to holomorphic sections. 

\bigskip

Now, we will establish the non-degenerate pairing for relative duality. As mentioned before it is enough to establish relative duality for smooth spectral curves, and since non-degeneracy is a local property it suffices to define the pairing for the trivial line bundle. Let  $U \subseteq X$ be an open subset trivialising any vector bundle. From Proposition {\ref{structuresheaf}} it follows that $1/\partial_\lambda \, p(\lambda), \lambda / \partial_\lambda \, p(\lambda), \ldots, \lambda^{n-1} / \partial_\lambda \, p(\lambda)$ defines a basis for the $\mathcal{O}_X(U)$-module $\mathcal{O}_S(R)(\pi^{-1}(U))$ since $(\partial_\lambda \, p(\lambda)) = R.$ In this setting, {\em Euler's formula} states
\[
\Tr_{S/X}(\lambda^j / \partial_\lambda \, p(\lambda)) = 
\begin{cases}
0 & \text{for $j = 0, \ldots, n-2$} \\
1 & \text{for $j = n-1$}.
\end{cases}
\]
The reader may find a proof of Euler's formula in \cite[Sec. III.2 pp. 127-128]{MR1215934}. From Euler's formula it follows that the image of $\pi_*(\mathcal{O}_S(R))$ under $\Tr_{S/X}$ defines a subsheaf of $\mathcal{O}_X,$ i.e., $\Tr_{S/X} : \pi_*(\mathcal{O}_S(R)) \to \mathcal{O}_X.$ This provides enough machinery to establish the non-degenerate pairing. Indeed, consider the $\mathcal{O}_X$-bilinear pairing \[\langle \, , \, \rangle : \pi_*(\mathcal{O}_S) \otimes \pi_*(\mathcal{O}_S(R)) \to \mathcal{O}_X\] defined by
\begin{equation}
\label{pairing}
\langle a, b \rangle = \Tr_{S/X}(a \otimes b).
\end{equation}
Over the local frames $1, \lambda, \ldots, \lambda^{n-1}$ and $1/\partial_\lambda \, p(\lambda), \lambda / \partial_\lambda \, p(\lambda), \ldots, \lambda^{n-1} / \partial_\lambda \, p(\lambda)$ for $\pi_*(\mathcal{O}_S)$ and $\pi_*(\mathcal{O}_S(R))$ respectively, it follows from Euler's formula that the intersection matrix has determinant $\pm 1$, which shows the pairing in (\ref{pairing}) is non-degenerate.

\subsection{BNR Corresondence}\label{secBNR}

Let $\pi : S \to X$ be a given smooth spectral curve. From Section \ref{sec2} recall that pushing forward a holomorphic line bundle $N \in \Pic(S)$ and the tautological section $\lambda$ defines a $\GL_n$-Higgs bundle with spectral curve $S$. It turns out that every $\GL_n$-Higgs bundle with spectral curve $S$ is induced by the pushforward of a  line bundle and the tautological section. Moreover, the line bundle determines the Higgs bundle up to isomorphism. This is known as the BNR-correspondence, or spectral curve correspondence. 

\begin{theorem}[{\cite[Proposition 3.6]{MR998478}}]
\label{BNR}
Suppose $\pi: S \to X$ is a smooth spectral curve. There is a bijective correspondence between isomorphism classes of Higgs bundles whose spectral curve is given by $S$ and isomorphism classes of holomorphic line bundles on $S$. The correspondence is given by associating a line bundle $N \in \Pic(S)$ to the Higgs bundle $(\pi_*(N), \pi_*(\lambda))$ where $\lambda : N \to \pi^*(L) \otimes N$ is the tautological section. 
\end{theorem}

The BNR-correspondence has been worked out in more cases, e.g., Hitchin computed the correspondence for $G_2$ see \cite[Theorem 2]{MR2354922}; and Mukhopadhyay and Wentworth computed the correspondence for $\operatorname{Spin}(N)$ and twisted $\operatorname{Spin}(N)$ see \cite[Theorem 1.1]{MR4230392}. Further, Scheinost and Schottenloher computed the spectral data for parabolic Higgs bundles for $G = \GL_n$ see \cite[Theorem 5.16]{MR1353317}; and Donagi and Gaitsgory considered {\em cameral data} instead of spectral data to give an analogous correspondence for arbitrary $G$, see \cite[Theorem 6.4]{MR1903115}. %

\section{Prym Varieties}\label{prymvarieties} Prym varieties are a special type of abelian variety canonically associated to branched covers between compact Riemann surfaces. Before defining Prym varieties we need to introduce the {norm map}. Suppose $f : \Sigma' \to \Sigma$ is a branched cover of compact Riemann surfaces, then the {\em norm map} is the assignment $\Nm_{\Sigma'/\Sigma} : \Pic(\Sigma') \to \Pic(\Sigma)$ defined on divisor classes by $\Nm_{\Sigma'/\Sigma}(\sum_{i=1}^n a_i x_i) = \sum_{i=1}^n a_if(x_i).$ It is easy to verify that the norm map defined on divisor classes preserves linear equivalence, which induces the map on the Picard group. In fact, since the norm map preserves the degree of the divisor restricting to $\Jac(\Sigma')$ defines a map $\Nm_{\Sigma'/\Sigma} : \Jac(\Sigma') \to \Jac(\Sigma).$ Now we may define a Prym variety.

\begin{definition}
Let $f : \Sigma' \to \Sigma$ be a branched cover between two compact Riemann surfaces. The {\em Prym variety} associated to $f : \Sigma' \to \Sigma$ is the connected component of the kernel of the norm map $\Nm_{\Sigma'/\Sigma}$ containing $\mathcal{O}_{\Sigma'}$, i.e., $\Prym(\Sigma', \Sigma) := (\ker(\Nm_{\Sigma'/\Sigma}))_0.$ 
\end{definition}

There are several  well-known properties associated to abelian varieties that we will take for granted where we refer the reader to the canonical references \cite{MR2062673, MR0282985} for complex abelian varieties. For instance, for a homomorphism $g : A_1 \to A_2$ between two abelian varieties, $\ker(g)$ is connected if and only if the dual map $g^\vee : A_2^\vee \to A_1^\vee$ is injective; and the dual homomorphism of  the norm map $\Nm_{\Sigma'/\Sigma} : \Jac(\Sigma') \to \Jac(\Sigma)$ is the pullback map, i.e., $f^* : \Jac(\Sigma) \to \Jac(\Sigma').$ Hence, $\Prym(\Sigma', \Sigma) = \ker(\Nm_{\Sigma'/\Sigma})$ if and only if $f^* : \Jac(\Sigma) \to \Jac(\Sigma')$ is injective. There is a necessary and sufficient condition to verify when $f^*$ is injective. Namely, $f^* : \Jac(\Sigma) \to \Jac(\Sigma')$ is not injective if and only if $f : \Sigma' \to \Sigma$ factorises through a cyclic \'etale cover of degree at least 2. The proof of this can be found in \cite[Proposition 4.3. pp. 337]{MR2062673}. 

\section{$\Sp_{2n}$-Higgs Bundles}\label{secsp2n}

Principal $\Sp_{2n}$-bundles correspond to rank $2n$ holomorphic vector bundles equipped with a symplectic form. Hence, a $\Sp_{2n}$-Higgs bundle is a triple $(E, \Phi, \omega)$ where $E$ is a rank $2n$ holomorphic vector bundle, $\omega : E \otimes E \to \mathcal{O}_C$ is a symplectic form, and $\Phi : E \to L \otimes E$ is a homomorphism such that $\omega(\Phi v, w) + \omega(v, \Phi w) = 0.$ To compute the linear system of divisors on $L^{2n}$ that parameterises spectral curves associated to $\Sp_{2n}$-Higgs bundles, we follow Hitchin's argument from \cite[Sec. 5.10]{MR0885778}. Let $A \in \mathfrak{sp}_{2n}$ have distinct eigenvalues $\lambda_i$ with corresponding eigenvectors $v_i \in \CC^{2n}$. Then,
\[
\lambda_i \omega(v_i, v_j) = \omega(A v_i, v_j) = -\omega(v_i, Av_j) = -\lambda_j \omega(v_i, v_j),
\] 
and hence, $(\lambda_i + \lambda_j) \omega(v_i, v_j) = 0.$  Thus, either $\lambda_i = -\lambda_j$ or $\omega(v_i, v_j) = 0.$ Since $\omega$ is non-degenerate it follows that eigenvalues occur in opposite pairs, and hence, the characteristic polynomial of $(E, \Phi, \omega)$ is given by
\begin{equation}
\label{sympleccharc}
p(\lambda) = \lambda^{2n} + a_2\lambda^{2n-2} + \cdots + a_{2n}.
\end{equation}
The $a_2, \ldots, a_{2n}$ defines a basis for the invariant polynomials on $\mathfrak{sp}_{2n}$. Allowing the $a_{2i}$ to vary defines a linear system of divisors and it follows by Bertini's theorem that a generic spectral curve is smooth. Note that from (\ref{sympleccharc}) a symplectic spectral curve $\pi: S \to X$ possesses a canonical involution defined by $\sigma(\lambda) = -\lambda$. Since the spectral curve is smooth, by the BNR-correspondence the Higgs bundle $(E, \Phi)$ is characterised by a holomorphic line bundle $N \to S$ where $E \cong \pi_*(N)$ and $\Phi$ is induced by the tautological section $\lambda :  N \to \pi^*(L) \otimes N$. Using the symplectic structure $\omega$ we will deduce further structure on $N$. To do so, we will define the notion of a {\em dual Higgs bundle}.

\begin{definition}
Let $(E, \Phi, \omega)$ be a $\Sp_{2n}$-Higgs bundle. The {\em dual Higgs bundle} (with respect to $\omega$) is the pair $(E^*, \Phi^t)$ where $E^*$ is the dual holomorphic vector bundle, and $\Phi^t : E^* \to L\otimes E^*$ is the adjoint of $\Phi$ with respect to $\omega.$
\end{definition}

Next, we require two short lemmas.

\begin{lemma}
\label{involutionpull}
Suppose that under the $BNR$-correspondence the Higgs bundle $(E, \Phi)$ corresponds to $N$. Then, $(E, -\Phi)$ corresponds to $\sigma^*(N)$.
\end{lemma}

\begin{lemma}
\label{dualhiggs}
Suppose that under the $BNR$-correspondence the Higgs bundle $(E, \Phi)$ corresponds to $N$. Then, the dual Higgs bundle $(E^*, \Phi^t)$ corresponds to $N^*(R)$ where $R$ denotes the ramification divisor of $\pi : S \to X.$
\end{lemma}

The proof of the first lemma is an exercise of unraveling definitions, and the proof of the second lemma follows from relative duality.  

\bigskip

The symplectic form induces an isomorphism $(E, -\Phi) \cong (E^*, \Phi^t)$, and by BNR induces an isomorphism $\theta : \sigma^*(N) \to N^*(R)$, which is unique up to scalar multiplication. Therefore, the $\Sp_{2n}$-Higgs bundle $(E, \Phi, \mu)$ defines a pair $(N, \theta)$ where $N \to S$ is a holomorphic line bundle and $\theta : \sigma^*(N) \to N^*(R)$ is an isomorphism.   

\bigskip

We must show that the isomorphism $\theta$ recovers the symplectic form establishing a one-to-one correspondence. The symplectic form $\omega : E \otimes E \to \mathcal{O}_X$ is unique up to scalar multiplication.\,\footnote{\, Using the fact $S$ is smooth a similar proof to that of Schur's lemma shows that the automorphisms of $(E, \Phi)$ are scalar multiples of the identity from which the claim follows.} 
Thus, it suffices to show that $\theta$ induces, in a natural way, a $\Phi$-compatible symplectic form. Let $U \subseteq X$ be a given open subset, and consider the $\mathcal{O}_X$-bilinear form defined by 
\begin{equation}
\label{symplectic form}
\mu(a, b) := \Tr_{S/X}(a \otimes \theta \sigma^* b)
\end{equation}
where $a,b \in \mathcal{O}_S(N)(\pi^{-1}(U)).$ Since $\theta$ is nowhere vanishing, and the pairing from relative duality is non-degenerate, it follows that $\mu$ is non-degenerate. Moreover, from $\Phi = \pi_* \lambda$ it is easily verified that $\mu$ is $\Phi$-compatible. We are left to verify that $\mu$ is skew-symmetric. Let $\bar{\sigma}$ be a linearisation of $\sigma$ on $\sigma^*(N^*)N^*(R).$ Then $\bar{\sigma}^2$ acts as $-1$ on the fibres and $\bar{\sigma}^* \theta = -\theta$. Since $\Tr_{S/X}(\sigma^*\cdot) = \Tr_{S/X}(\cdot)$,
%
\begin{equation}
\label{skew}
\mu(b,a) = \Tr_{S/X} \sigma^*(b \otimes  \theta \sigma^* a) = \Tr_{S/X} (\sigma^*b \otimes  \bar{\sigma}^*\theta a) = -\Tr_{S/X} (a \otimes  \theta \sigma^* b) = -\mu(a, b).
\end{equation}
Thus, (\ref{skew}) shows that the bilinear form $\mu$ is skew-symmetric. 
%
%
%
Therefore, the spectral data for $(E, \Phi, \omega)$ is precisely the pair $(N, \theta)$ where $N \in \Pic(S)$ and $\theta \in \HH^0(S, \sigma^*(N^*)N^*(R))$ is nowhere vanishing. We will now reformulate the spectral data into the language of Prym varieties. 

\bigskip

The compact Riemann surface $\bar{S} := S/\sigma$ defines a spectral curve over $X$ defined by the polynomial
\[
\eta^n + b_1 \eta^{n-1} + \cdots + b_n
\]  
where $\eta$ is the tautological section of $L^2$ and $b_i := a_{2i}$. We denote the natural map from $\bar{S}$ to $X$ by $q:\bar{S} \to X$, then if $p : S \to \bar{S}$ is the natural projection map we see $\pi = q\circ p.$ By the adjunction formula we may consider $\theta$ as an isomorphism $ \theta : N\sigma^*(N) \to \pi^*(L^{2n-1}).$ Now, $p^* \Nm_{S/\bar{S}}(N) \cong N\sigma^*(N),$ and hence, 
$
p^* \Nm_{S/\bar{S}}(N) \cong p^*q^* L^{2n-1}.
$
Since $\pr:S \to \bar{S}$ is a ramified double cover, $p^*: \Jac(\bar{S}) \to \Jac(S)$ is injective, and thus, $\Nm_{S/\bar{S}}(N) \cong q^*(L^{2n-1}).$ Therefore, $N$ belongs to a torsor of the Prym variety $\Prym(S, \bar{S})$, and we have shown that isomorphism classes of $\Sp_{2n}$-Higgs bundles with smooth spectral curve $\pi:S \to X$ is canonically identified to a torsor of the Prym variety $\Prym(S, \bar{S}).$ 

\section{$\SO_{2n+1}$-Higgs Bundles}\label{secso2n+1}

Principal $\SO_{m}$-bundles over $X$ correspond to rank $m$ holomorphic vector bundles equipped with a symmetric non-degenerate bilinear form. Hence, a $\SO_{m}$-Higgs bundle is a triple $(E, \Phi, Q)$ where $E$ is a rank $m$ holomorphic vector bundle with $\det(E) \cong \mathcal{O}_X$, $Q : E \otimes E \to \mathcal{O}_X$ is a non-degenerate $\mathcal{O}_X$-bilinear symmetric form, and $\Phi : E \to L \otimes E$ is a homomorphism such that $Q(\Phi(v), w) + Q(v, \Phi(w)) = 0.$ Let $(E, \Phi, Q)$ be a $\SO_{2n+1}$-Higgs bundle. By a similar argument as the one used in Section \ref{secsp2n} we see that generically eigenvalues occur in opposite pairs and zero is an eigenvalue. Hence, the characteristic polynomial of $(E, \Phi, Q)$ is given by
\begin{equation}
\label{SO2n+1}
\det(\lambda - \pi^*(\Phi)) = \lambda(\lambda^{2n} + a_2\lambda^{2n-2} + \ldots + a_{2n}).
\end{equation}

Note that the corresponding spectral curve will always be singular since $L$ has positive degree. However, 
\[
p(\lambda) = \lambda^{2n} + a_2\lambda^{2n-2} + \ldots + a_{2n}
\]
defines a symplectic spectral curve, so we may assume the $a_{2i} \in \HH^0(X, L^i)$ define a smooth symplectic spectral curve. Denote the canonical involution by $\sigma$. From (\ref{SO2n+1}) it is evident that the zero eigenbundle establishes a correspondence between generic $\SO_{2n+1}$-Higgs bundles and generic $\Sp_{2n}$-Higgs bundles. The correspondence is simple away from the singularities of the $\SO_{2n+1}$ spectral curve, and we will largely recall Hitchin's argument from \cite[Sec. 4]{MR2354922}. To extend the correspondence over the singularities, we will use a Hecke modification argument and utilise local calculations in a frame. Also, in establishing the correspondence over the singularities we will introduce more spectral data, which amounts to dualising the Prym variety from the $\Sp_{2n}$ classification realising Langlands duality.

\subsection{Correspondence away from the Singularities}{\label{1}}

Since the singularities of the spectral curve defined by (\ref{SO2n+1}) occur precisely when $\lambda = 0$ and $a_{2n} = 0$, let $D = (a_{2n})$ and consider $X_0 = X-D.$ We will establish the correspondence now over $X_0.$ Note, we will consider objects defined over $X$ and then restrict our attention to $X_0.$ Suppose $(V, \phi, Q)$ is a $\SO_{2n+1}$-Higgs bundle and consider the subbundle $V_0$ induced by the kernel of the sheaf map $\phi : \mathcal{O}_X(V) \to \mathcal{O}_X(L \otimes V).$ Generically the eigenvalue $\lambda = 0$ has multiplicity 1 so it is clear that the subbundle $V_0 \subset V$ defines a holomorphic line bundle. Thus, the quotient bundle $E = V/V_0$ defines a rank $2n$ holomorphic vector bundle over $X$. The Higgs field $\phi : V \to L \otimes V$ induces a Higgs field $\Phi : E \to L \otimes E$ whose characteristic polynomial is given by $p(\lambda).$ Next, we consider $\omega : E \otimes E \to L$ defined by 
\begin{equation}
\label{symplecticform}
\omega(v, w) = Q(\phi(\tilde{v}), \tilde{w})
\end{equation}
where $\tilde{v}, \tilde{w}$ are lifts of the sections to $v, w$ to $V$ respectively. Since $V_0$ is induced by the kernel of $\phi$ it is easy to see that $\omega$ is well-defined, and it is easy to verify that $\omega$ is $\Phi$-compatible and skew-symmetric. Moreover, since $\ker(\phi_y) = (V_0)_y$ and $Q_y$ is non-degenerate for every $y \in X_0$, it follows that $\omega$ is non-degenerate over $X_0.$ 

\begin{construction}
\label{firstone}
Suppose $(V, \phi, Q)$ is a generic $\SO_{2n+1}$-Higgs bundle over $X$ whose characteristic polynomial is of the form $\lambda p(\lambda)$ where $p(\lambda)$ defines a smooth symplectic spectral curve. The kernel of the Higgs field induces a holomorphic subbundle $V_0 \subset V$ that has rank 1 and $E = V/V_0$ defines a rank $2n$ holomorphic vector bundle over $X$. Moreover, the map $\omega : E \otimes E \to L$ defined by $\omega(v, w) = Q(\phi (\tilde{v}), \tilde{w})$ defines a symplectic form on $X_0 = X - D$ where $D = (a_{2n}).$ Further, the Higgs field $\phi$ induces a Higgs field $\Phi : E \to L \otimes E$ whose characteristic polynomial is given by $p(\lambda)$ and $\omega$ is $\Phi$-compatible. Thus, $(E, \Phi, \omega)$ defines a $\Sp_{2n}$-Higgs bundle over $X_0$ whose characteristic polynomial is the restriction of $p(\lambda).$
\end{construction}

\begin{remark}
The triple $(E, \Phi, Q)$ established in Construction \ref{firstone} is globally defined over $X$. To show $(E, \Phi, Q)$ defines a $\Sp_{2n}$-Higgs bundle over $X$ we are left to show that $\omega$ is non-degenerate over $D.$
\end{remark}

To provide the reverse construction we need to recover the zero eigenbundle. We will construct a canonical zero eigensection $v_0 \in \HH^0(X_0, V_0)$, and establish a natural isomorphism to a known line bundle. Consider the skew symmetric form $\alpha$ of $\Wedge^2 V^*$ defined by $\alpha(v,w) = Q(\phi v, w).$ Let $y \in X_0$ be given, and let $e^0, \ldots, e^{2n}$ be an orthonormal frame for $V^*$ over $X_0$ with respect to $Q.$ In the eigenbundle decomposition $V = V_0 \oplus V_{\lambda_1} \oplus V_{-\lambda_1} \oplus \cdots \oplus V_{\lambda_n} \oplus V_{-\lambda_n}$ the orthogonal form $Q$ pairs opposite eigenvalues, hence we may assume without loss of generality that 
\[
\alpha = i\lambda_1 e^1 \wedge e^2 + \cdots + i\lambda_n e^{2n-1} \wedge e^{2n}.
\]
Note that $\alpha^n = n!i^n \lambda_1 \cdots \lambda_n e^1 \wedge \cdots \wedge e^{2n}.$ Denote the dual orthonormal frame with respect to $Q$ by $e_0, \ldots, e_{2n}$. Let $\nu = e^0 \wedge \cdots \wedge e^{2n}$ denote the $\SO_{2n+1}$-invariant volume form for $V^*.$ Consider now the section $v_0 = i^n \lambda_1\cdots\lambda_n e_0$ of $L^n \otimes V$. Since $e^j(e_i) = \delta_i{}^j$ notice that $\alpha^n = n!i_{v_0}(\nu)$. Now, a straightforward calculation shows $\phi \cdot i_{v_0} = i_{\phi(v_0)}(\nu) + i_{v_0}(\phi \cdot \nu)$, and since $\phi \cdot \nu = 0$ and $\phi \cdot \alpha = 0$ it follows that $i_{\phi(v_0)}(\nu) = 0.$ However, $\nu$ is non-degenerate, and thus, $\phi(v_0) = 0.$ Hence, $v_0$ is the desired canonical zero eigensection and since $v_0$ is nowhere vanishing it follows that $V_0 \cong L^{-n}$ over $X_0.$ Thus, we may use $L^{-n}$ to recover the zero eigenbundle, and we note that $Q(v_0, v_0) = a_{2n}.$ 

\bigskip

Suppose $(E, \Phi, \omega)$ is a $\Sp_{2n}$-Higgs bundle over $X$ with characteristic polynomial $p(\lambda)$ where $\omega$ is $L$-valued. Consider the rank $2n+1$ holomorphic vector bundle $V = L^{-1}E \oplus L^{-n}$ and Higgs field $\phi : V \to L \otimes V$ defined by $\phi(v,w) := (\Phi v, 0).$ Next, consider the symmetric $\mathcal{O}_X$-bilinear form $Q : V \otimes V \to \mathcal{O}_X$ defined by 
\[
Q((v,w), (v,w)) := \omega(v, \Phi v) + w^2 a_{2n}
\] 
that extends to every section by the polarisation identity. The reader may easily verify that $Q$ is $\phi$-compatible, and since over $X_0$ the section $a_{2n}$ is nowhere vanishing and $\omega$ is non-degenerate, it follows that $Q$ is non-degenerate. Also, since $X_0$ is open the condition $\det(V) \cong \mathcal{O}_{X_0}$ is vacuous. Thus, over $X_0$ the triple $(V, \phi, Q)$ defines a $\SO_{2n+1}$-Higgs bundle whose characteristic polynomial is given by $\lambda p(\lambda).$

\begin{construction}
\label{secondone}
Suppose $(E, \Phi, \omega)$ is a generic $\Sp_{2n}$-Higgs bundle where $\omega$ is $L$-valued and with characteristic polynomial $p(\lambda)$ that defines a smooth symplectic spectral curve. Then, $V = L^{-1}E \oplus L^{-n}$ defines a rank $2n+1$ holomorphic vector bundle, and $\Phi : E \to L \otimes E$ induces a Higgs field $\phi : V \to L \otimes V$ whose characteristic polynomial is given by $\lambda p(\lambda).$ Moreover, the $\mathcal{O}_X$-bilinear symmetric form $Q : V \otimes V \to \mathcal{O}_X$ defined by \[Q((v,w), (v,w)) = \omega(v, \Phi v) + w^2 a_{2n}\]and the polarisation identity is $\phi$-compatible and non-degenerate over $X_0 = X-D$ where $D = (a_{2n}).$  Thus, $(V, \phi, Q)$ defines a $\SO_{2n+1}$-Higgs bundle on $X_0$ whose characteristic polynomial is the restriction of $\lambda p(\lambda).$
\end{construction}

Unravelling the constructions in Construction \ref{firstone} and Construction \ref{secondone} shows that the constructions are mutual inverses. This establishes a one-to-one correspondence away from the singularities of $S$.

\subsection{Extending the Correspondence over the Singularities}{\label{2}}

Now, we extend the correspondence over the singularities. We continue using the same notation as in Sections \ref{secso2n+1} and \ref{1}. Before extending the correspondence over the singularities we will make one assumption, which is valid since $L$ is basepoint-free.

\begin{assumption}
\label{assumption}
We will assume that the sections $a_{2n-2}$ and $a_{2n}$ have no common zeros. 
\end{assumption}

One can check in local coordinates using the fact that $S$ is smooth that the zeros of $a_{2n}$ are simple. Let $(E, \Phi, \omega)$ be the $\Sp_{2n}$-Higgs bundle over $X_0$ that is associated to $(V, \phi, Q).$ The triple $(E, \Phi, \omega)$ is defined over $X$ and to show $(E, \Phi, \omega)$ defines a $\Sp_{2n}$-Higgs bundle over $X$ it suffices to show that $\omega$ is non-degenerate over $D$. To show this we first require a quick lemma. 

\begin{lemma}
Let $x \in X$ be a zero of $\det(\Phi).$ If $\dim(\ker(\Phi_x))>1$, then $\ord_x(\det(\Phi)) > 1.$ 
\end{lemma}

\begin{proof}
Suppose $\dim(\ker(\Phi_x))>1$. Then we may choose two linearly independent vectors $e_1, e_2$ in $\ker(\Phi_x)$ and extend to a basis $e_1, e_2, \ldots, e_{2n}$ for $E_x$. We may choose a local coordinate $z$ centred at $x$ that trivialises $E$ and extend the basis to a local frame $e_1(z), e_2(z), \ldots, e_{2n}(z)$ for $E$. Then, observe that
\[
\Phi(z) e_1(z) \wedge \Phi(z) e_2(z) \wedge \cdots \wedge \Phi(z) e_{2n}(z) = \det(\Phi)(z) \, e_1(z) \wedge \cdots \wedge e_{2n}(z). 
\]
The left-hand-side vanishes to at least second order, and hence, $\ord_x(\det(\Phi)) >1.$
\end{proof}

Since each zero of $a_{2n} = \det(\Phi)$ is simple it follows that $\dim(\ker(\Phi_x)) = 1$ for each $x \in D.$ By Assumption \ref{assumption} we see $a_{2n-2}(x) \ne 0.$ Thus, the zero-generalised eigenspace of $\phi_x$ is three-dimensional, which we denote by $A$. Let $B$ denote the sum of the remaining generalised eigenspaces. By the primary decomposition theorem
\begin{equation}
\label{decomp1}
V_x = A \oplus B.
\end{equation}

Since $Q_x$ pairs opposite eigenvalues and the generalised eigenspaces are $\phi_x$-invariant it follows that the decomposition in (\ref{decomp1}) is $Q_x$-invariant and $\phi_x$-invariant. Thus, the restriction  of $\phi_x$ to $A$ defines a nilpotent homomorphism $\alpha : A \to L_x \otimes A.$ By fixing a basis of $A$ we may identify $\mathfrak{so}(A) \cong \mathfrak{so}_3(\CC).$ Then, since $\alpha$ is nilpotent it follows that $\alpha$ belongs to one of three conjugacy classes, which may be represented by Jordan canonical form. 

\bigskip

First suppose $\alpha = 0_3$. Then, $\phi_x = 0_3 \oplus M$ where $M : B \to L_x \otimes B$ denotes the restriction of $\phi_x$ to $B$. However, we may assume without loss of generality that the first basis element spans $(V_0)_x$. Hence, $\Phi_x = 0_2 \oplus M$ and this implies that $\dim(\ker(\Phi_x))>1$, which is a contradiction. 

Next, suppose that by a suitable choice of basis $e_1, e_2, e_3$
\[
\alpha = \begin{bmatrix} 0 & 1 & 0 \\ 0 & 0 & 0 \\ 0 & 0 & 0 \end{bmatrix}.
\]
Then, $\alpha(e_1) = 0$, $\alpha(e_2) = e_1$, and $\alpha(e_3) = 0.$ However, this implies that $Q_x$ is degenerate. Indeed,
\[
Q_x(e_1, e_2) = Q_x(\alpha(e_2), e_2) = -Q_x(e_2, \alpha(e_2)) = -Q_x(e_1, e_2) 
\]
which shows $Q_x(e_1, e_2) = 0,$ and moreover,
\begin{align*}
Q_x(e_1, e_1) = Q_x(\alpha(e_2), e_1) = -Q_x(e_2, \alpha(e_1)) = 0 \\
Q_x(e_1, e_3) = Q_x(\alpha(e_2), e_3) = -Q_x(e_2, \alpha(e_3)) = 0.
\end{align*}
Therefore, by a suitable choice of basis $\alpha$ is given by
\[
\alpha = \begin{bmatrix} 0 & 1 & 0 \\ 0 & 0 & 1 \\ 0 & 0 & 0 \end{bmatrix}.
\]
Thus, $(V_0)_x = \ker(\phi_x),$ which allows us to prove that $(E, \Phi, \omega)$ defines a $\Sp_{2n}$-Higgs bundle over $X$.

\begin{proposition}
The skew-symmetric form $\omega : E \otimes E \to L$ is non-degenerate over $X$. Hence, $(E, \Phi, \omega)$ defines a $\Sp_{2n}$-Higgs bundle over $X$ whose spectral curve is $S$.  
\end{proposition}

\begin{proof}
Let $x \in D$ be given. Suppose that $\omega_x(v, w) = 0$ for every $w \in E_x.$ Then, $Q_x(\phi_x(\tilde{v}), \tilde{w}) = 0$ for every $\tilde{w} \in V_x.$ Since $Q_x$ is non-degenerate it follows that $\tilde{v} \in \ker(\phi_x) = (V_0)_x$, i.e., $v = 0.$
\end{proof}

Thus, Construction \ref{firstone} defines a global construction. However, Construction \ref{secondone} does not define a global construction since $Q(v_0, v_0) = a_{2n}$, and hence, $Q_x = 0$ for every $x \in D.$ Contrary to other authors, to extend the construction over $D$ we take advantage of Hecke transformations. This avoids the use of extension classes and simplifies the proofs to be more self-contained. We will characterise $\mathcal{O}_X(V)$ as a subsheaf of $\mathcal{M}_X(L^{-1}E \oplus L^{-n}) := \mathcal{O}_X(L^{-1}E \oplus L^{-n}) \otimes_{\mathcal{O}_X} \mathcal{M}_X$. More specifically, we will allow sections of $V$ to have poles at $D$ such that $\phi$ and $Q$ extends over $D$ holomorphically, and such that $Q$ is non-degenerate and $\phi$-compatible. In what proceeds, let $U_1 = X_0$ and $V_1 := E.$   

\bigskip

Consider the short exact sequence
\[
0 \to V_0 \xrightarrow{\iota} V \xrightarrow{\pr} V_1 \to 0.
\]
Dualising the sequence and using $Q$ to canonically identify $V \cong V^*$ gives
\[
0 \to V_1^* \xrightarrow{\pr^*} V \xrightarrow{\iota^*} V_0^* \to 0.
\]
Let $U_2$ denote the disjoint union of sufficiently small open discs over each point in $D$. By local triviality we can choose a splitting $\tau : (V_0)^*\vert_{U_{2}} \to V\vert_{U_{2}}$, i.e., over $U_2$ the vector bundles $V$ and $V_1^* \oplus V_2^*$ are isomorphic with explicit isomorphism
\[
\psi_2 : (V_1^* \oplus V_0^*) \vert_{U_2} \, \ni (a,b) \mapsto \pr^*(a) + \tau(b) \in V \vert_{U_2}.
\]
Over $U_1$ the form $Q_0 := Q\vert_{V_0}$ is non-degenerate and defines an isomorphism $V_0 \cong V_0^*.$ By non-degeneracy $V_0 \cap V_0^{\perp} = 0$ over $U_1$ and by naturally identifying $V_0^\perp \cong \Ann(V_0) \cong V_1^*$ over $U_1$ we obtain an isomorphism $V \vert_{U_1} \cong (V_1^* \oplus V_0) \vert_{U_1}$ that is given by
\begin{equation}\label{psi1}
\psi_1 : (V_1^* \oplus V_0) \vert_{U_1} \ni (a,b) \mapsto \pr^*(a) + \iota(b) \in V \vert_{U_1}.
\end{equation} 
Hence, we have obtained two isomorphisms for the vector bundle $V$ over two different open sets. Now, we will compute the difference of the isomorphisms over $U_{12} := U_{1} \cap U_2,$ i.e., we will compute
\[
\psi_{21} := \psi_2^{-1} \circ \psi_1 : (V_1^* \oplus V_0) \vert_{U_{12}} \to (V_1^* \oplus V_0^*) \vert_{U_{12}}.
\]
First, consider $\psi_{21}(a,0).$ Since $\psi_1(a,0) = \pr^*(a) = \psi_2(a,0)$ it follows that $\psi_{21}(a,0) = (a,0).$ Now, consider $\psi_{21}(0,b) = \psi_2^{-1}(\iota(b)).$ To understand $\psi_2^{-1}(\iota(b))$ let $u : V_0 \to V_1^*$ be the homomorphism $u(b) = \pr_1(\psi_2^{-1}(\iota(b)))$ where $\pr_1 : V_1^* \oplus V_0^* \to V_1^*$ denotes the naturally projection map. Since $\iota^* \circ \iota = Q_0$ it follows that $\psi_2(0, Q_0(b)) = \iota(b)$. Hence, $\psi_{21}(0,b) = (u(b), Q_0(b))$ and thus,
\begin{equation}\label{trans}
\psi_{21}(a,b) = (a + u(b), Q_0(b)).
\end{equation}
A straightforward calculation shows that the inverse map is given by
\[
\psi_{12}(a,b) = (a-u(Q_0^{-1}(b)), Q_0^{-1}(b)).
\]
The transition maps $\psi_{12}$ and $\psi_{21}$ are defined away from $D$. However, we may extend the maps over $D$ by allowing the sections to have poles. Recall that $(Q_0)_x = 0$ for each $x \in D.$ Hence, $(V_0)_x \subset (V_0^{\perp})_x \cong (V_1^*)_x.$ Moreover, it follows from (\ref{trans}) that $\iota_x = \pr^*_x \circ u_x$, and hence, $u_x : (V_0)_x \to (V_1^*)_x$ is injective for each $x \in D.$ This provides enough machinery to characterise $\mathcal{O}_X(V)$ as a subsheaf of $\mathcal{M}_X(V_1^* \oplus V_0)$.

\begin{proposition}
\label{sheaf}
The sheaf $\mathcal{O}_X(V)$ can be realised as a subsheaf of $\mathcal{M}_X(V_1^* \oplus V_0)$ that has simple poles along $D$ with residues valued in
\[
\Gamma_x = \{(-i_x(w), w) \, \vert \, w \in (V_0)_x\}, \ \ \ x \in D
\]
where $i_x : (V_0)_x \to (V_1^*)_x$ is injective.
\end{proposition}

\begin{proof}
Recall that $V \cong V_1^* \oplus V_0$ away from $D$ with isomorphism given by $\psi_1$ in (\ref{psi1}), and since $V_0 \subset V$ is a degenerate subbundle with respect to $Q$ the isomorphism does not extend over $D$. However, allowing for poles over $D$ we may use $\psi_1^{-1}$ to realise $\mathcal{O}_X(V)$ as a subsheaf of $\mathcal{M}_X(V_1^* \oplus V_0)$. Since $\psi_1^{-1} = \psi_{12} \circ \psi_2^{-1}$, and over $U_2$ the map $\psi_2^{-1}$ is an isomorphism it suffices to characterise the subsheaf
\[
\psi_{12} \, \mathcal{O}_X(V_1^* \oplus V_0^*) \subset \mathcal{M}_X(V_1^* \oplus V_0).
\]
Fix $x \in D$ and choose a local coordinate $z$ centred at $x$ that trivialises $L$. Since $a_{2n}$ has a simple zero at $x$ we may assume without loss of generality that $a_{2n} = z.$ Moreover, recall that there is a canonical section $v_0$ of $L^n \otimes V_0$ such that $Q(v_0, v_0) = z.$ The section $v_0$ defines an isomorphism $V_0 \cong L^{-n}$ away from $x$, and it follows that in this trivialisation $Q_0^{-1} = \frac{1}{z}.$ Therefore, given a local holomorphic section $(a(z), b(z))$ of $V_1^* \oplus V_0$ we see 
\[
\psi_{12}(a(z), b(z)) = \left ( a(z) - \frac{u(z)b(z)}{z}, \frac{u(z)b(z)}{z} \right).
\]
Since $u(0) = u_x$ is injective the result follows. 
\end{proof}

It follows from Proposition \ref{sheaf} that if $(V, \phi, Q)$ is a generic $\SO_{2n+1}$-Higgs bundle with induced $\Sp_{2n}$-Higgs bundle $(E, \Phi, \omega)$, then $\mathcal{O}_X(V)$ may be realised as the subsheaf of $\mathcal{M}_X(L^{-1}E \oplus L^{-n})$ whose sections are holomorphic away from $D = (\det(\Phi))$ with simple poles along $D$ whose residues belong to $\Gamma_x = \{(-i_x(w), w) \, \vert \, w \in L_x^{-n}\}$ for each $x \in D$ where $i_x : L_x^{-n} \to L_x^{-1}E_x$ is injective. Note that we have used the symplectic form $\omega : E \otimes E \to L$ to identify $E^* \cong L^{-1}E.$

\bigskip

Now, we extend the reverse construction over the singularities using Hecke modifications following Abe \cite{MR2448083}. Suppose $(E, \Phi, \omega)$ is a $\Sp_{2n}$-Higgs bundle with characteristic polynomial $p(\lambda)$, and let $(V, \phi, Q)$ be the $\SO_{2n+1}$-Higgs bundle defined over $X_0$ obtained by Construction \ref{secondone}. We may extend $V$ over $D$ by considering $\mathcal{O}_X(V)$ as a subsheaf of $\mathcal{M}_X(L^{-1}E \oplus L^{-n})$ comprised of sections that are holomorphic away from $D$ with simple poles along $D$ whose residues are valued in $\Gamma_x = \{(-i_x(w), w) \, \vert \, w \in L_x^{-n}\}$ for each $x \in D$ where $i_x : L_x^{-n} \to L_x^{-1}E_x$ is injective. The homomorphisms $\{i_x\}_{x \in D}$ characterise the vector bundle $V$ completely. Moreover, for $(V, \phi, Q)$ to define a $\SO_{2n+1}$-Higgs bundle over $X$ we require $\phi$ and $Q$ to extend holomorphically over $D$ such that $Q$ is $\phi$-compatible. Thus, we will determine the family of homomorphisms $\{i_x\}_{x \in D}$ that give our desired extension. Of course, we also require $V$ to have trivial determinant and an orientation. For the orientation there are two choices. However, the automorphism, $-\id$, exchanges the two orientations, so the orientation is unique up to isomorphism. 

\begin{lemma}
The holomorphic vector bundle $V$ has trivial determinant.
\end{lemma}

\begin{proof}
By construction the holomorphic vector bundle $V$ fits into a short exact sequence
\[
0 \to L^{-n} \to V \to E \to 0.
\]
Thus, $\det(V) \cong L^{-n} \otimes \det(E).$ Further, since the symplectic form $\omega : E \otimes E \to L$ is non-degenerate, $\omega^n$ defines a nowhere vanishing section of $\det(E)^* \otimes L^n.$ Therefore, $\det(E) \cong L^n$, and hence, $\det(V) \cong \mathcal{O}_X.$
\end{proof}

Hence, we only need to modify $i_x$ so that $\phi$ and $Q$ extend holomorphically over $D$ such that $Q$ is non-degenerate. 

\begin{lemma}
\label{extendphi}
The Higgs field $\phi$ extends holomorphically over $D$ if and only if $i_x(L_x^{-n}) = \ker(\Phi_x)$ for each $x \in D.$ 
\end{lemma}

\begin{proof}
Fix $x \in D.$ Choose a local coordinate $z$ of $X$ centred at $x$. We extend the homomorphism $i_x : L_x^{-n} \to L_x^{-1}E_x$ to a local section $i(z)$ of $\Hom(L^{-n}, L^{-1}E).$ Let $w(z)$ be a local section of $L^{-n}$ and consider
\[
\phi(z)\left ( -\frac{i(z)w(z)}{z}, \frac{w(z)}{z} \right ) = \left ( -\frac{1}{z}\Phi(z)i(z)w(z), 0 \right ).
\]
Hence, $\phi(z)$ extends holomorphically over $z = 0$ if and only if $\Phi_x(i_x(w)) = 0$ where $w = w(0).$ In other words, $\phi(z)$ extends holomorphically over $x$ if and only if $i_x(L_x^{-n}) = \ker(\Phi_x).$
\end{proof}

Lemma \ref{extendphi} determines the homomorphisms $\{i_x\}$ up to scale. Now, we choose $\{i_x\}$ such that $i_x(L_x^{-n}) = \ker(\Phi_x)$ for each $x \in D.$ Now, we will appropriately scale the $i_x$ so that $Q$ extends holomorphically over $D$. If both entries in $Q$ are holomorphic, then the output will be holomorphic. Hence, we will check the cases when one input is holomorphic and the other is meromorphic, and for when both entries are meromorphic. We adopt the notation used in the proof of Lemma \ref{extendphi}. Consider the following
\begin{align}
\label{holomero}
\begin{split}
Q \left ( \left ( \frac{-i(z)w(z)}{z}, \frac{w(z)}{z} \right ), (-i(z)w(z),w(z)) \right )  \\
= \frac{1}{z}\left (\omega(i(z)w(z), \Phi(z)i(z)w(z)) + w(z)^2 a_{2n}(z) \right ).
\end{split}
\end{align}
Since $a_{2n}(z)$ has a simple zero at $z = 0$ and $\Phi_x(i_x(w)) = 0$ it follows that (\ref{holomero}) is holomorphic. Now, consider 
\begin{align}
\label{meromero}
\begin{split}
Q \left ( \left ( \frac{-i(z)w(z)}{z}, \frac{w(z)}{z} \right ),  \left ( \frac{-i(z)w(z)}{z}, \frac{w(z)}{z} \right ) \right )  \\
= \frac{1}{z^2}\omega(i(z)w(z), \Phi(z)i(z)w(z)) + \frac{w(z)^2}{z^2} a_{2n}(z).
\end{split}
\end{align}
By writing $\Phi(z) = \Phi_x + z\Phi_x' + z^2A(z)$ where $A(z)$ is a holomorphic function and $a_{2n}(z) = za_{2n}'(0) + z^2B(z)$ where $B(z)$ is a holomorphic function it follows that (\ref{meromero}) simplifies to
\begin{align*}
\begin{split}
Q \left ( \left ( \frac{-i(z)w(z)}{z}, \frac{w(z)}{z} \right ),  \left ( \frac{-i(z)w(z)}{z}, \frac{w(z)}{z} \right ) \right )  \\
= \frac{1}{z}\left( \omega(i_x w, \Phi_x'i_xw) + w^2a_{2n}'(0) \right ) \mod \mathcal{O}_x.
\end{split}
\end{align*}
Therefore, $Q$ extends holomorphic over $D$ if and only if 
\begin{equation}
\label{condition}
\omega(i_x w, \Phi_x'i_xw) = -w^2a_{2n}'(0)
\end{equation} 
for each $x \in D.$

\begin{proposition}
\label{Qextends}
We may choose $\{i_x\}$ so that $(\ref{condition})$ holds. In other words, we may choose $\{i_x\}$ so that $Q$ extends holomorphically over $D$.
\end{proposition}

\begin{proof}
Fix $x \in D$. Let $(U, z)$ be local coordinates centred at $x$ where $U$ is biholomorphic to a disc. Also, let $\ell := \ell(z)$ be a trivialising section of $L$ over $U.$ Locally, the eigenvalues of $\Phi(z)$ are given by $\pm \lambda_1(z) \ell, \ldots, \pm \lambda_n(z) \ell.$ Since $a_{2n-2}(0) \ne 0$ the zero generalised eigenspace of $\Phi_x$ is two-dimensional, and we may assume without loss of generality that $\lambda_i(0) = 0$ if and only if $i=1$. We may factor the characteristic polynomial as
\[
p(\lambda, z) = (\lambda^2 + \alpha)\,\hat{p}(\lambda, z)
\]
where $\alpha = -\lambda_1(z)^2$ and $\hat{p}(\lambda, z) := (\lambda^2 - \lambda_2(z)^2) \cdots (\lambda^2 - \lambda_n(z)^2).$ Let $\beta = (-1)^{n-1}\lambda_2(z)^2\cdots \lambda_n(z)^2.$ Then, $\alpha \beta = a_{2n}$ and $\beta(0) = a_{2n-2}(0),$ and it follows that 
\begin{equation}
\label{alphadiff}
\alpha'(0) = \frac{a_{2n}'(0)}{a_{2n-2}(0)}.
\end{equation}
Let $E_0 \subset E\vert_U$ be the subbundle induced by the kernel of $\Phi(z)^2 + \alpha(z)\ell^2.$ Notice that $(E_0)_x$ is the zero generalised eigenspace of $\Phi_x$, which is two-dimensional, so $E_0$ has rank 2. Since $\ker(\Phi_x) \subset (E_0)_x$, let $e_1, e_2$ be a basis for $(E_0)_x$ such that $\Phi_x e_1 = 0$ and $\Phi_x e_2 = e_1 \ell_x.$ Extend $e_2$ to a nowhere vanishing section $e_2(z)$ and set $e_1(z) := \ell^{-1} \Phi(z)e_2(z).$ The sections $e_1(z), e_2(z)$ form a frame for $E_0.$ By definition of $E_0$
\[
\Phi(z)e_1(z) = \ell^{-1} \Phi(z)^2e_2(z) = -\ell \alpha(z) e_2(z).
\]
Thus, in the frame
\[
\Phi(z) = \begin{bmatrix}
0 & \ell \\  -\ell \alpha(z) & 0 
\end{bmatrix}.
\]
Recall that $i_x(L_x^{-n}) = \ker(\Phi_x)$ so $i_x(\ell_x^{-n}) = ue_1\ell_x^{-1}$ for some $u \ne 0.$ We may extend $i_x$ to a local section of $L^{n-1}E$ by defining $i(z) := u \ell^{n-1} e_1(z).$ Hence, $i(z)\ell^{-n} = u\ell^{-1}e_1(z)$, and thus,
\begin{equation}
\label{yes}
\Phi(z)i(z)\ell^{-n} = \begin{bmatrix}
0 & \ell \\  -\ell \alpha(z) & 0 
\end{bmatrix} 
\begin{bmatrix}
u\ell^{-1} \\ 0
\end{bmatrix} = 
\begin{bmatrix}
0 \\ -u \alpha(z)
\end{bmatrix}.
\end{equation}
Differentiating (\ref{yes}) and evaluating at $z=0$ it follows from (\ref{alphadiff}) that \begin{equation}\label{formula} \Phi'_x(i_x(\ell_x^{-n})) = -\frac{u a_{2n}'(0)}{a_{2n-2}(0)}e_2.\end{equation} Now, wish wish to modify $i_x$ so that
\[
\omega_x(i_x \ell_x^{-n}, \Phi_x' i_x \ell_x^{-n}) = -a_{2n}'(0).
\]
By substituting in (\ref{formula}) and $i_x \ell_x^{-n} = u \ell_x^{-1} e_1$ it follows that we require $i_x$ to satisfy
\begin{equation}
\label{requirement}
u^2\omega_x(e_1, e_2) = \ell_x a_{2n-2}(0).
\end{equation}
Since $\omega_x(e_1, e_2) \ne 0$ and $\ell_x a_{2n-2}(0) \ne 0$ we may choose $u$ such that $(\ref{requirement})$ holds.
\end{proof}

\begin{remark}
The condition in (\ref{requirement}) can be reformulated in coordinate-free terms. Indeed, from $i_x(L_x^{-n}) = \ker(\Phi_x)$ we see $\Phi_x i_x = 0.$ Moreover, we may choose a vector $j_x \in L_x^{n-2}E_x$ such that $\Phi_x j_x = i_x$, e.g., $j_x = \ell_x^{n-2}e_2.$ The vector $j_x$ is unique up to $\ker(\Phi_x),$ and condition (\ref{requirement}) becomes
\[
\omega_x(i_x, j_x) = a_{2n-2}(x)
\] 
where $j_x \in L_x^{n-2} \otimes \ker(\Phi_x^2)/\ker(\Phi_x)$ satisfies $\Phi_xj_x = i_x.$
\end{remark}

Therefore, we may choose $\{i_x\}$ so that $Q$ extends holomorphically over $D$. Note that from Proposition \ref{Qextends} the choice of $i_x$ is unique up to sign. We may extend this condition to sections by decreeing that \[u^2\omega(e_1(z), e_2(z)) = \ell a_{2n-2}(z).\] By replacing $ue_1(z)$ and $ue_2(z)$ with $e_1(z)$ and $e_2(z)$ respectively we may assume that
\begin{equation}
\label{QQ}
\omega(e_1(z), e_2(z)) = \ell a_{2n-2}(z).
\end{equation} 
We are left to show that $Q_x$ is non-degenerate for each $x \in D.$ 

\bigskip

To prove that $Q_x$ is non-degenerate we will compute the determinant in a local frame. Fix $x \in D$. We adopt all notation from the proof of Proposition \ref{Qextends} and we impose condition (\ref{formula}). Let $E_1 \subset E\vert_U$ be the subbundle induced by the kernel of $\hat{p}(\Phi(z), z).$ Since $(E_1)_x$ is the sum of the non-zero generalised eigenspaces we see that $(E_1)_x$ has dimension $2n-2,$ and hence, $E_1$ has rank $2n-2.$ By the Cayley-Hamilton theorem $E\vert_U = E_0 \oplus E_1.$ The composition is $\Phi$-invariant and  orthogonal with respect to $\omega$ since $\omega$ pairs opposite eigenvalues, and each generalised eigenspace is invariant under $\Phi.$ Under this decomposition notice that
\[
(L^{-1}E \oplus L^{-n}) \vert_U = (L^{-1}\vert_U E_0 \oplus L^{-n} \vert_U) \oplus L^{-1} \vert_U E_1.
\] 
Moreover, since $i_x(L_x^{-n}) = \ker(\Phi_x) \subset L^{-1}_x(E_0)_x$ notice that \[\Gamma_x \subset L^{-1}_x(E_0)_x \oplus L_x^{-n}.\] Hence, set $V_1 = L^{-1}\vert_U E_1$ and let $V_0$ be the holomorphic vector bundle whose sheaf of holomorphic sections is the subsheaf \[\mathcal{O}_U(V_0) \subset \mathcal{M}_U(L^{-1} \vert_U E_0 \oplus L^{-n} \vert_U)\] comprised of sections that are holomorphic away from $x$ with a simple pole at $x$ with residue valued in $\Gamma_x$. Then, we obtain a decomposition $V \vert_U \, = V_0 \oplus V_1.$ We claim that this decomposition is orthogonal with respect to $Q$. By continuity it suffices the prove that the decomposition $V_y = (V_0)_y \oplus (V_1)_y$ is orthogonal with respect to $Q_y$ where $y \ne x.$ Let $(u,v) \in (V_0)_y$ and $w \in (V_1)_y.$ By the polarisation identity it follows that
\[
Q_y((u,v), w) = \omega_y(u, \Phi_yw).
\]
Since the decomposition $E_y = (E_0)_y \oplus (E_1)_y$ is $\Phi_y$-invariant and orthogonal with respect to $\omega_y$ it follows that $Q_y((u,v), w) = 0$, which proves the claim. Therefore, we may decompose $Q\vert_U \, = Q_0 \oplus Q_1$ where $Q_i : V_i \otimes V_i \to \mathcal{O}_U.$

\begin{proposition}
The form $Q_x : V_x \otimes V_x \to \CC$ is non-degenerate for each $x \in D.$
\end{proposition}

\begin{proof}
Fix $x \in D.$ By the foregoing discussion it suffices to show that both $(Q_0)_x$ and $(Q_1)_x$ are non-degenerate. To see that $(Q_1)_x$ is non-degenerate suppose $(Q_1)_x(u,v) = 0$ for every $u \in (V_1)_x.$ By the polarisation identity it follows that $\omega_x(u, \Phi_x v) = 0$ for every $u \in (V_1)_x.$ Since $\omega_x$ is non-degenerate it follows that $\Phi_x v = 0$, and since $(E_1)_x$ is the sum of the non-zero generalised eigenspaces, $v = 0$. Therefore, we are left to show that $(Q_0)_x$ is non-degenerate, which we will show using a local calculation. Set $e_3(z) := \ell^{-n}$, then observe that 
\[
f_1(z) := \left( -\frac{\ell^{-1}e_1(z)}{z}, \frac{e_3(z)}{z} \right ), \ \ \ f_2(z) := \left( \ell^{-1}e_1(z), e_3(z) \right ), \ \ \ f_3(z) := (\ell^{-1}e_2(z), 0)
\]
defines a frame for $V_0.$ Then, 
\begin{align*}
Q_0(f_1(z), f_1(z)) &= \frac{\ell^{-1}}{z^2}\omega(e_1(z), \Phi(z)\ell^{-1}e_1(z)) + \frac{a_{2n}(z)}{z^2} \\
&= \frac{1}{z^2}(a_{2n}(z) - \ell^{-1}\alpha(z)\omega(e_1(z), e_2(z))) \\
&=  \frac{1}{z^2}(a_{2n}(z) - \alpha(z)a_{2n-2}(z))
\end{align*}
where the last equality follows from $(\ref{QQ})$. Notice that $\alpha(z)a_{2n-2}(z) = a_{2n}(z) - \alpha(z)\Gamma(z)$ for some holomorphic function $\Gamma(z)$ where $\Gamma(0) = 0$. Since $ \alpha(z)\Gamma(z)$ vanishes at $z=0$ to at least order 3 it follows that $(Q_0)_x(f_1(0), f_1(0)) = 0.$ A similar calculation shows that $(Q_0)_x(f_2(0), f_2(0)) = 0.$ Next,
\begin{align*}
Q_0(f_3(z), f_3(z)) &= \omega(\ell^{-1}e_2(z), \Phi(z)\ell^{-1}e_2(z)) \\
&= -\ell^{-1}\omega(e_1(z), e_2(z)) \\
&= -a_{2n-2}(z)
\end{align*}
where the last equality again follows from (\ref{QQ}). Evaluating at $z =0$ gives $Q_x(f_3(0), f_3(0)) = -a_{2n-2}(0).$ By the polarisation identity it follows that
\begin{align*}
Q_0(f_1(z), f_2(z)) &= \frac{1}{z} ( a_{2n}(z) - \omega(e_1(z)\ell^{-1}, \Phi_1(z)e_1(z)\ell^{-1})) \\
&= \frac{1}{z} ( a_{2n}(z) + \ell^{-1} \alpha(z) \omega(e_1(z), e_2(z))) \\
&= \frac{1}{z} ( a_{2n}(z) + \alpha(z)a_{2n-2}(z)).
\end{align*}
Since $a_{2n}$ has a simple zero at $z=0$ and $\alpha$ has a zero of order 2 at $z = 0$ we obtain $(Q_0)_x(f_1(0), f_2(0)) = a_{2n}'(0).$ Similarly, by the polarisation identity it is straightforward to show that $(Q_0)_x(f_1, f_3) = 0$ and $(Q_0)_x(f_2, f_3) = 0.$ Therefore, in the given basis
\[
(Q_0)_x = \begin{bmatrix} 0 & a_{2n}'(0) & 0 \\
a_{2n}'(0) & 0 & 0 \\
0 & 0 & -a_{2n-2}(0)
\end{bmatrix}.
\]
The determinant of $(Q_0)_x$ is $a_{2n}'(0)^2a_{2n-2}(0),$ which is non-zero, and hence, $(Q_0)_x$ is non-degenerate. 
\end{proof}

This the desired global construction using Hecke modifications. We formally state the construction below.

\begin{construction}
Suppose $(E, \Phi, \omega)$ is a generic $\Sp_{2n}$-Higgs bundle with characteristic polynomial $p(\lambda)$, and set $D := (a_{2n}).$ Let $(V, \phi, Q)$ be the $\SO_{2n+1}$-Higgs bundle defined over $X - D$ obtained by applying Construction \ref{secondone}. To extend the bundle over $D$ we identify the sheaf $\mathcal{O}_X(V)$ with the subsheaf of $\mathcal{M}_X(L^{-1}E \oplus L^{-n})$ comprised of sections that are holomorphic away from $D$ with simple poles along $D$ whose residues are valued in $\Gamma_x = \{(-i_xw, w) \, \vert \, w \in L_x^{-n}\}$ for each $x \in D$ where $i_x : L_x^{-n} \to L_x^{-1}E_x$ is injective. Moreover, by choosing $i_x$ such that $i_x(L_x^{-n}) = \ker(\Phi_x)$ the Higgs field $\phi$ extends holomorphically over $D$. In fact, by choosing $i_x$ such that 
\[
\omega_x(i_x w, \Phi_x'(i_x w)) = -w^2a_{2n}'(0)
\] 
for every $x \in D$, the form $Q$ extends holomorphically over $D$ such that the extension is non-degenerate and $\phi$-compatible. Thus, $(V, \phi, Q)$ defines a $\SO_{2n+1}$-Higgs bundle over $X$ and the characteristic polynomial is given by $\lambda p(\lambda).$
\end{construction}

Finally, we determine exactly when different  choices of $\{i_x\}$ define isomorphic $\SO_{2n+1}$-Higgs bundles. 

\begin{proposition}
\label{degree}
Let $(E, \Phi, \omega)$ be a generic $\Sp_{2n}$-Higgs bundle where $\omega$ is $L$-valued. Suppose $(V, \phi, Q)$ and $(V', \phi', Q')$ are two $\SO_{2n+1}$-Higgs bundles associated to $(E, \Phi, \omega)$ defined by $\{i_x\}$ and $\{i_x'\}$ respectively. Then $(V, \phi, Q) \cong (V', \phi, Q')$ if and only if $i_x = i_x'$ for every $x \in D$ or $i_x = -i_x'$ for every $x \in D.$ 
\end{proposition}

\begin{proof}
Let $\psi : (V, \phi, Q) \to (V', \phi', Q')$ be an isomorphism. Since $\psi$ commutes with the Higgs fields there is an induced isomorphism $\psi_2 : L^{-n} \to L^{-n}$. Hence, $\psi_2 = c_2 \id$ since $\psi_2$ is an automorphism of a holomorphic line bundle over a compact Riemann surface. Next, there is an induced automorphism on the quotient bundle $L^{-1}E$, i.e., $\psi_1 : L^{-1}E \to L^{-1}E$. Since $(L^{-1}E, \Phi)$ is a Higgs bundle with a smooth spectral curve the automorphisms are non-zero scalar multiples of the identity, i.e., $\psi_1 = c_1\id.$ Thus,
\[
\psi = \begin{bmatrix}
c_1 & * \\ 0 & c_2
\end{bmatrix}.
\] 
Away from $D$ both $V$ and $V'$ are isomorphic to $L^{-1}E \oplus L^{-n}$, so away from $D$ we have $* = 0.$ By identifying $\mathcal{O}_X(V)$ and $\mathcal{O}_X(V')$ as subsheaves of $\mathcal{M}_X(L^{-1}E \oplus L^{-n})$ it follows by continuity that $* = 0$ everywhere, and hence,
\[
\psi = \begin{bmatrix}
c_1 & 0 \\ 0 & c_2
\end{bmatrix}.
\]
Since $\psi$ is an isomorphism, $\psi$ preserves the space of residues for each $x \in D$. Let $x \in D$ be given and suppose $w \in L_x^{-n}.$ Then,
$
\psi_x((-i_x w, w)) = \left (-\frac{c_1}{c_2} i_x (c_2w), c_2 w \right )$. Notice that $i'_x = t_xi_x$ where $t_x = \pm1$ and it follows that $t_x = \frac{c_1}{c_2}$, which is constant and independent of $x \in D$. Therefore, either $c_1 = c_2$, which holds if and only if $i_x = i_x'$ for every $x \in D$, or $c_1 = -c_2$, which holds if and only if $i_x = -i_x'$ for every $x \in D.$ 

\end{proof}

\subsection{Spectral Data for $\SO_{2n+1}$-Higgs Bundles}
\label{SO_{2n+1} spec data}

Using the correspondence between generic $\SO_{2n+1}$-Higgs bundles and $\Sp_{2n}$-Higgs bundles we will now describe the spectral data for generic $\SO_{2n+1}$-Higgs bundles. We adopt the notation from the previous section, i.e., $(V, \phi, Q)$ is a $\SO_{2n+1}$-Higgs bundle with characteristic polynomial $\lambda p(\lambda)$, $(E, \phi, \omega)$ is a $\Sp_{2n}$-Higgs bundle with characteristic polynomial $p(\lambda)$ that defines a smooth spectral curve $S$, and $(V, \phi, Q)$ is determined by $(E, \Phi, \omega)$ and homomorphisms $\{i_x\}_{x \in D}.$ Since $\omega : E \otimes E \to L$ is $L$-valued the spectral data is slightly different to that found in Section \ref{secsp2n}. Namely, the isomorphisms $\theta$ will be between different line bundles. 

\bigskip

Suppose $E \cong \pi_*(N)$ where $N \in \Pic(S).$ The reader may easily verify using the same arguments from Section \ref{secsp2n} that the spectral data for $(E, \Phi, \omega)$ is $(N, \theta)$ where $\theta : \sigma^*(N) \to N^*(R) \pi^*(L)$ is an isomorphism, or, equivalently, $\theta \in \HH^0(S, \sigma^*(N^*)N^*(R)\pi^*(L))$ is nowhere vanishing. Our goal is to encode the data of the homomorphisms $\{i_x\}$ into $\theta.$ Identifying $D$ with $\lambda = 0$ in $S$ defines a divisor in $S$, which we denote by $D_S.$ Let $y \in D_S$ where $\pi(y) = x.$ Then, in a local coordinate chart $(U, z)$ centred at $x$ recall that we may locally factor the characteristic polynomial by \[p(\lambda, z) = (\lambda^2 + \alpha)\hat{p}(\lambda, z)\] where $\alpha(z) = -\lambda_1(z)^2$ and $\alpha(0) = 0$. The point $y$ belongs to the curve $\pi_0: S_0 \to U$ defined by $\lambda^2 +\alpha$. Let $F \subset E \vert_U$ denote the rank 2 subbundle induced by the kernel of the sheaf map $\Phi(z)^2 + \alpha(z)\ell^2$ where $\ell := \ell(z)$ is a trivialising section of $L$ over $U$. Then, $F \cong (\pi_0)_*(N_0)$ where $N_0 := N \vert_{S_0}$ and $\theta_y \in (N_0)^{-2}_yL_xR_y$. Therefore, to encode the data of $i_x$ into the section $\theta$ we have distilled the problem to studying the isomorphism in the $\Sp_2$-Higgs bundle $(F, \Phi, \omega)$ defined over $U$ with characteristic polynomial $p(\lambda, z) := \lambda^2 -z \ell^2$ where $F = \pi_*(N_0)$ and $\Phi$ and $\omega$ are restricted to $F$. For simplicity we relabel $N_0$ by $N$. The fibre of $F$ at $x$ is given by
\[
F_x \cong \mathcal{O}_U(F) \otimes \mathcal{O}_x/\mathfrak{m}_x.
\] 
where $\mathfrak{m}_x \subset \mathcal{O}_x$ is the maximal ideal of holomorphic functions vanishing at $x$. Letting $\lambda = w\ell$ where $w$ is a holomorphic coordinate centred at $y$ it follows that $z = w^2$, and hence, 
\[
F_x \cong \mathcal{O}(N) \otimes \mathcal{O}_y/\mathfrak{m}^2_y.
\]
Thus, $n, nw \mod(w^2)$ defines a basis for $F_x.$ Since $p(\Phi(z), z) = 0$ we see $\Phi(z)^2 = z\ell^2$, and so $\Phi(z) = w\ell$. It follows that $\Phi(z)n = nw\ell$ and $\Phi(z)nw = 0$, so we let $e_1 = nw$ and $e_2 = n.$ Now, $i_x = u\ell_x^{n-1} e_1$ and $j_x = u\ell_x^{n-2} e_2$ for some $u \ne 0,$ and
\[
\omega_x(i_x, j_x) = a_{2n-2}(x).
\] 
Then, it follows that in the basis
\begin{equation}
\label{conditionone}
u^2\omega_x(e_1, e_2) = a_{2n-2}(0)\ell_x.
\end{equation}
Since $\theta_y \in N_{y}^{-2}R_yL_x$ we may write $\theta_y = \frac{a n^{-2}\ell_x}{w}$ for some $a\ne 0.$  Then, since $\omega_x(e_1, e_2) = \Tr_{S/U}(e_1 \otimes \theta_y e_2)$ it follows that
\begin{equation}
\label{conditiontwo}
\omega_x(e_1, e_2) = 2 a \ell_x.
\end{equation}
Combining $(\ref{conditionone})$ and $(\ref{conditiontwo})$ gives $2au^2 = a_{2n-2}(0).$ Next, recall that by the adjunction formula $\mathcal{O}_S(R) \cong \pi^*(L^{2n-1})$, which is locally determined by 
\[
\frac{1}{w} \mapsto \frac{\partial_\lambda p(\lambda, z)}{w} \mod{w^2}.
\]
Then, since $\frac{\partial_\lambda p(\lambda, z)}{w} = 2a_{2n-2}(z)\ell^{2n-1} \mod{w^2}$ it follows that
\begin{equation}
\label{j}
\theta_y = 2aa_{2n-2}(0)n^{-2} \ell_x^{2n}.
\end{equation}
Finally, notice that $\theta_yj_x^2 \in L_x^{4n-4}$ where by (\ref{j})
\[
\theta_y j_x^2 = 2au^2a_{2n-2}(0) \ell_x^{4n-4}.
\]
However, $2au^2 = a_{2n-2}(0)$ and it follows that $\theta_y j_x^2 = a_{2n-2}(x)^2,$ i.e., 
\[
\theta_y = \left ( \frac{a_{2n-2}(x)}{j_x} \right )^2.
\]
Therefore, the family $\{j_x\}_{x \in D}$, which determines the family $\{i_x\}_{x \in D}$, defines a square root of the homomorphisms $\{\theta_y\}_{y \in D_S}$. 

\bigskip

Now, we may reformulate this data into the language of torsors of Prym varieties. In what proceeds, let $\bar{S} := S/\sigma$. For each $m \in \ZZ$ define the $\Prym(S, \bar{S})$-torsor
\[
\mathcal{T}_m := \left \{(U, [\chi]) \, : \, U \in \Pic(S), \, \chi \in \HH^0(S, \sigma^*(U^*)U^*\pi^*(L^m)) \, \text{nowhere vanishing} \right \}.
\]
Notice that $\mathcal{T}_{2n}$ characterises the spectral data for generic $\Sp_{2n}$-Higgs bundles in this case. Next, we define 
\[
\bar{\mathcal{T}}_{2n} = \mathcal{T}_{2n} \times \{\{j_x\}/_{\pm 1} \, : \, \Phi_xj_x = i_x, \, \omega_x(i_x, j_x) = a_{2n-2}(x)\},
\]
which characterises the spectral data for generic $\SO_{2n+1}$-Higgs bundles. The relationship between $\SO_{2n+1}$ and $\Sp_{2n}$-Higgs bundles is described by the projection map $\pr : \bar{\mathcal{T}}_{2n} \to \mathcal{T}_{2n}$, which has degree $2^{2n\deg(L) - 1}.$ To see this, recall that the zeros of $a_{2n}$ are simple, of which, there are $\deg(L^{2n}) = 2n\deg(L)$ total. Also, for each zero there are a choice of two vectors characterising the $\SO_{2n+1}$-Higgs bundle, and by Proposition \ref{degree} the choices are unique up to a total change of sign, i.e., $\{i_x\}$ and $\{-i_x\}$ define isomorphic Higgs bundles. We wish to compute the number of connected components of $\bar{\mathcal{T}}_{2n}$ and determine the underlying abelian variety. To do so, consider the squaring map $\bar{s} : \mathcal{T}_n \to \bar{\mathcal{T}}_{2n}$ defined by
\begin{equation}
\label{liftofsquaring}
\bar{s}(U, [\chi]) = (U^2, [\chi^2], [\chi \vert_{D_S} ]/_{\pm 1}).
\end{equation} 
The map defined in (\ref{liftofsquaring}) is the lift of the regular squaring map $s : \mathcal{T}_n \to \mathcal{T}_{2n}.$ It is rudimentary to verify that for $(U_1, \chi_1), (U_2, \chi_2) \in \mathcal{T}_n$ we have $U_1^2 \cong U_2^2$ if and only if $U_1 \cong U_2 \otimes A$ where $A \in \Jac(S)[2]$. Note that $\sigma^*(A) \cong A$ so there is a lift $\bar{\sigma} : A \to A$. After scaling if necessary we may assume without loss of generality that $\bar{\sigma}$ is involutive and $\chi_1 = \chi_2 \otimes \bar{\sigma}.$ Since $\chi_1 \vert_{D_S} = \chi_2 \vert_{D_S}$ it follows that $\bar{\sigma} \vert_{D_S} = 1$, so $B = A/\bar{\sigma}$ defines a holomorphic line bundle on $\bar{S}$ where $B^2 \cong \mathcal{O}_{\bar{S}}$. Since $A \cong p^*(B)$ it follows that the induced map 
\[
\bar{s} : \mathcal{T}_n / p^*\Jac(\bar{S})[2] \to \bar{\mathcal{T}}_{2n}
\]
defines an injection. By a careful application of the Riemann-Hurwitz formula the reader may verify that the induced squaring map
\[
s = \pr \circ \bar{s} : \mathcal{T}_n / p^*\Jac(\bar{S})[2] \to \mathcal{T}_{2n}
\]
has degree $2^{2n\deg(L)-2}.$ It follows that $\bar{\mathcal{T}}_{2n}$ has two connected components that are isomorphic as algebraic varieties to $\mathcal{T}_n / p^*\Jac(\bar{S})[2].$ However, $\mathcal{T}_n / p^*\Jac(\bar{S})[2]$ is a torsor of $\Prym(S, \bar{S}) / p^*\Jac(\bar{S})[2]$, which establishes the first result in Theorem \ref{main theorem}. However, note that $  \Prym(S, \bar{S}) / p^*\Jac(\bar{S})[2] \cong \Prym(S, \bar{S})^\vee.$

\begin{lemma}
There is an isomorphism $\Prym(S, \bar{S})^\vee \cong \Prym(S, \bar{S}) / p^*\Jac(\bar{S})[2]$. 
\end{lemma}

\begin{proof}
Consider the short exact sequence
\[
0 \to \Prym(S, \bar{S}) \to \Jac(S) \xrightarrow{\Nm_{S/\bar{S}}} \Jac(\bar{S}) \to 0.
\]
Dualising the sequence gives
\[
0 \to \Jac(\bar{S}) \xrightarrow{p^*} \Jac(S) \to \Prym(S, \bar{S})^\vee \to 0
\]
and hence, $\Jac(S) / p^* \Jac(\bar{S}) \cong \Prym(S, \bar{S})^\vee.$ Restricting to $\Prym(S, \bar{S})$ defines an isomorphism $\Prym(S, \bar{S})^\vee \cong \Prym(S, \bar{S})/(p^*\Jac(\bar{S}) \cap \Prym(S, \bar{S}))$. To describe $p^*\Jac(\bar{S}) \cap \Prym(S, \bar{S})$ let $b \in \Jac(\bar{S})$. Then $p^{-1}(b) = a +\sigma(a)$ where $p(a) = b$, and we see that $p(p^{-1}(b)) = 2b$. Thus, $p^{-1}(b) \in p^*\Jac(\bar{S}) \cap \Prym(S, \bar{S})$ if and only if $b \in \Jac(\bar{S})[2]$ and the result follows. 
\end{proof}

The duality in the Prym varieties between the spectral data for generic $\Sp_{2n}$-Higgs bundles and $\SO_{2n+1}$-Higgs bundles is precisely Langlands duality in this setting. 

\section{$\SO_{2n}$-Higgs Bundles}\label{sec5}

Let $(E, \Phi, Q)$ be a $\SO_{2n}$-Higgs bundle. Similar to $\Sp_{2n}$-Higgs bundles, the characteristic polynomial is even. However, contrary to the $\Sp_{2n}$ case if $A \in \mathfrak{so}_{2n}$ has characteristic polynomial $p(x) = x^{2n} + a_2x^{2n-2} + \cdots + a_{2n}$, then the characteristic coefficients $a_2, \ldots, a_{2n}$ do not form an invariant homogenous basis for $\mathfrak{so}_{2n}.$ The polynomial $a_{2n}$ is the square of a polynomial $p_n$ called the {\em Pfaffian}, and the polynomials $a_2, \ldots, a_{2n-2}, p_n$ form a basis. Hence, the characteristic polynomial of $(E, \Phi, Q)$ is of the form
\begin{equation}
\label{evensopoly}
p(\lambda) = \lambda^{2n} + a_2\lambda^{2n-2} + \cdots + a_{2n-2}\lambda^2 + p_n^2
\end{equation}
where $a_i \in \HH^0(X, L^{2i})$ and $p_n \in \HH^0(X, L^n).$ The spectral curve $\pi : S \to X$ defined by $(\ref{evensopoly})$ has the canonical involution $\sigma(\lambda) = -\lambda.$ Now, due to the presence of the Pfaffian, the base locus of the linear system obtained by allowing the $a_{2i}$ to vary in (\ref{evensopoly}) is not empty. Indeed, the base locus is precisely the fixed points of $\sigma.$ By Bertini's theorem a generic spectral curve is smooth away from the base locus. We will now assume that the sections $a_{2n-2}$ and $p_n$ have no common zeros, which is valid since $L$ is basepoint-free. Then, under this assumption it can be easily verified in local coordinates that the fixed points of $\sigma(\lambda) = -\lambda$ define ordinary double points of the spectral curve. Further, if $n>1$ a counting argument shows that a generic spectral curve $S$ is an irreducible scheme.

\subsection{$\SO_{2n}$ BNR-correspondence}
\label{SO_{2n} BNR-correspondence}

Suppose now that $(E, \Phi, Q)$ is a generic $\SO_{2n}$-Higgs bundle with generic spectral curve $\pi : S \to X$. Since the spectral curve is not smooth we cannot apply the BNR correspondence. To overcome this, consider the canonical complex analytic normalisation $\nu : \hat{S} \to S$. Let $\hat{\pi} = \pi \circ \nu.$ Although the normalisation $\hat{\pi} : \hat{S} \to X$ is not necessarily a spectral curve, we will adapt the BNR correspondence to $\hat{S}$. Note that away from the singularities $\nu : \hat{S} \to S$ defines a biholomorphism.  Since $\hat{S}$ is a compact Riemann surface, consider the holomorphic line bundle $A$ induced by the kernel of the sheaf map $\nu^*(\lambda) - \hat{\pi}^*(\Phi).$ Let $R$ denote the ramification divisor of $\hat{\pi}.$ Then, the line bundle $A(R) \to \hat{S}$ is analogous to the line bundle from the BNR correspondence. We claim that $\hat{\pi}_*(A(R)) \cong E$ with isomorphism 
\begin{equation}
\label{tracemapso}
\Tr_{\hat{S}/X} : \hat{\pi}_*(\mathcal{O}_S(A(R))) \to \mathcal{O}_X(E).
\end{equation} 
Since away from the singularities $\hat{S}$ and $S$ are biholomorphic, the trace map defines an isomorphism as seen in the smooth case. Therefore, we only need to verify that the trace map defines an isomorphism near the singularities of $S$, i.e., near the zeros of $p_n.$ To do so, we will show that it is enough to verify the isomorphism in the $\SO_2$ case. In a sufficiently small neighbourhood $U$ of $x$ the spectral curve is the disjoint union of two irreducible spectral curves, i.e., $S\vert_{\pi^{-1}(U)} \, = S_1 \cup S_2,$ where $S_1$ is a degree $2$ curve containing the singularity and $S_2$ is a smooth degree $2n-2$ curve. Suppose $p_1(\lambda)$ and $p_2(\lambda)$ are the polynomials defining $S_1$ and $S_2$ respectively, then $p(\lambda) = p_1(\lambda)p_2(\lambda).$ From this factorisation, the Higgs bundle has a decomposition 
\begin{equation}
\label{decomp}
(E, \Phi) \vert_U \, = (E_1, \Phi_1) \oplus (E_2, \Phi_2)
\end{equation}
where the characteristic polynomial of $(E_i, \Phi_i)$ is given by $p_i(\lambda)$ for $i=1,2.$ To see this, let $E_1$ and $E_2$ are the subbundles of $E$ induced be the kernel of the sheaf maps $p_1(\Phi)$ and $p_2(\Phi)$ respectively. Then, by definition $E_1$ and $E_2$ have characteristic polynomials $p_1(\lambda)$ and $p_2(\lambda)$ respectively. Moreover, $E_1$ and $E_2$ are $\Phi$-invariant, and hence, $\Phi_i := \Phi \vert_{E_i}$ define Higgs fields for $i=1,2$. To obtain the decomposition in (\ref{decomp}) we left to show $E \vert_U \cong E_1 \oplus E_2.$ For $y \ne x$, it is clear that $E_y \cong (E_1)_y \oplus (E_2)_y$, and hence, we only need to show $E_x \cong (E_1)_x \oplus (E_2)_x.$ Since the eigenvalues of $\Phi$ are generically distinct, $\rank(E_i) = \rank(p_i)$ for $i = 1,2$, and thus, $\dim(E_x) = \dim((E_1)_x) + \dim((E_2)_x).$ To see that $(E_1)_x \cap (E_2)_x = \{0\}$, let $z$ be local coordinates on $U$ centred at $x$. Then, $(E_i)_x$ is contained in the generalised eigenspaces of $\Phi_x$ corresponding to the zeros of $p_i(\lambda, 0)$ for $i=1,2$. However, since $S_1$ and $S_2$ are disjoint, $p_1(\lambda, 0)$ and $p_2(\lambda, 0)$ have distinct zeros, and thus, $(E_1)_x \cap (E_2)_x = \{0\},$ which verifies the decomposition in (\ref{decomp}). In fact, the decomposition is orthogonal with respect to $Q$. Indeed, for $y \ne x$ the eigenvalues of $\Phi_y$ are distinct and it is clear that $(E_1)_y$ and $(E_2)_y$ are orthogonal with respect to $Q_y$. By continuity, $(E_1)_x$ and $(E_2)_x$ are orthogonal with respect to $Q_x$, which verifies the claim.

\bigskip

Since the normalisation is local and $S_2$ is smooth, the normalisation of $S$ over $\pi^{-1}(U)$ is given by $\hat{S_1} \cup S_2$, where $\nu_1 : \hat{S_1} \to S_1$ is the normalisation of $S_1.$ Let $N_1$ and $N_2$ denote the restriction of $A(R)$ to $\hat{S_1}$ and $S_2$ respectively. Let $\hat{\pi_1} := \hat{\pi} \vert_{\hat{S_1}}$, and notice that $\hat{\pi} \vert_{S_2} = \pi \vert_{S_2}$, which we denote by $\pi_2.$ Then, over $U$ there is a canonical isomorphism $\hat{\pi}_*(A(R)) \cong (\hat{\pi_1})_*(N_1) \oplus (\pi_2)_*(N_2).$ Over $U$, the trace map in (\ref{tracemapso}) decomposes as a sum $\Tr_{S\vert_{\pi^{-1}(U)/U}} = \Tr_{\hat{S_1}/U} \oplus \Tr_{S_2/U}.$ Therefore, it suffices to show that both $\Tr_{\hat{S_1}/U}$ and $\Tr_{S_2/U}$ are isomorphisms. However, $\Tr_{S_2/U}$ is an isomorphism from the BNR correspondence, and thus, we have reduced to verifying the $\SO_2$ case. Assume now that $(E, \Phi, Q)$ is a generic $\SO_2$-Higgs bundle with characteristic polynomial $p(\lambda) = \lambda^2 + q^2$ where $q \in \HH^0(X, L)$ has simple zeros. Let $\mu = iq$ so that $p(\lambda) = (\lambda -\mu)(\lambda+\mu),$ and suppose $x \in X$ is a zero of $q.$ Thus, $\lambda = \mu$ or $\lambda = -\mu$, which corresponds to the two sheets in $\hat{S}.$ It follows that $\Phi_x$ has eigenvalue $0$ of multiplicity 2, and hence, $\Phi_x$ is nilpotent. However, the only nilpotent element of $\mathfrak{so}_2$ is 0, and thus, $\Phi_x = 0.$ Therefore, locally $\Phi = \mu\beta$ where $\beta$ has eigenvalues $1,-1$. In an appropriate local frame, $\beta = \operatorname{diag}(1,-1)$, and $\Phi = \operatorname{diag}(\mu, -\mu).$  We will show in this frame that $\Tr_{S/X}$ is an isomorphism locally about $x.$ Note that locally $\hat{\pi}$ is unramified, so $\mathcal{O}(R) \cong\mathcal{O}.$ Let $V_1$ and $V_2$ denote the two sheets of the normalisation such that $\lambda \vert_{V_1} = \mu \id$ and $\lambda \vert_{V_2} = -\mu \id.$  Hence, $(\lambda - \Phi) \vert_{V_1} = \diag(0, 2\mu)$, and $(\lambda - \Phi) \vert_{V_2} = \diag(-2\mu, 0).$ Thus, in the local frame $A \vert_{V_1} \cong \langle (1,0) \rangle$ and $A \vert_{V_2} \cong \langle (0,1) \rangle.$ In the local frame a simple calculation shows that the trace map in $(\ref{tracemapso})$ is locally given by
\[
(f,g) = \begin{cases} (f,0) & \text{on $V_1$} \\ (0,g) & \text{on $V_2$} \end{cases} \mapsto (f,0) + (0,g) = (f,g),
\]
which is an isomorphism. Therefore, the trace map in (\ref{tracemapso}) is an isomorphism. 

\bigskip

The same argument as in the outline of the proof of the BNR correspondence in Section \ref{secBNR} shows that if $N \in \Pic(\hat{S})$ then $(\hat{\pi}_*(N), \hat{\pi}_*(\nu^*\lambda))$ defines a Higgs bundle on $X$ whose spectral curve is $S$; and that the assignments are mutual inverses. That is, Higgs bundles on $X$ with spectral curve $S$ naturally correspond to holomorphic line bundles on $\hat{S}$. We will assume $n>1$ so that $S$ is irreducible, and consequently, automorphisms of Higgs bundles with spectral curve $S$ are scalar multiples of the identity. Now, we will show that given $(E, \Phi, Q)$ with spectral curve $S$ and corresponding line bundle $N \in \Pic(\tilde{S})$, the orthogonal form $Q$ induces an isomorphism $\theta : \hat{\sigma}(N) \to N^*(R)$. Then, we will reconstruct a $\Phi$-compatible orthogonal form from the isomorphism $\theta$, both of which are unique up to scale by our assumption. 

\bigskip

Similar to Lemma \ref{involutionpull} the Higgs bundle $(E, -\Phi)$ corresponds to the line bundle $\hat{\sigma}^*(N) \in \Pic(\hat{S}).$ Also, similar to Lemma \ref{dualhiggs} it follows from relative duality that the dual Higgs bundle $(E^*, \Phi^t)$ with respect to $Q$ corresponds to the line bundle $N^*(R) \in \Pic(\hat{S}).$ Then, the orthogonal form $Q : E \otimes E \to \mathcal{O}_X$ defines an isomorphism $(E, -\Phi) \cong (E^*, \Phi^t).$ Hence, $Q$ induces an isomorphism $\theta : \hat{\sigma}^*(N) \to N^*(R),$ or, equivalently, a nowhere vanishing section of $\hat{\sigma}^*(N^*)N^*(R).$ Now, to construct a $\Phi$-compatible orthogonal form let $U \subseteq X$ be a given open set with $a, b \in \mathcal{O}_{\hat{S}}(N)(\hat{\pi}^{-1}(U))$, and consider the $\mathcal{O}_X$-bilinear form $\mu : E \otimes E \to \mathcal{O}_X$ defined by
\begin{equation}
\label{orthogform}
\mu(a,b) := \Tr_{\hat{S}/X}(a \otimes \theta \hat{\sigma}^*(b)).
\end{equation}

Since $\theta$ is nowhere vanishing and the form in (\ref{orthogform}) is the pairing from relative duality it follows that $\mu$ is non-degenerate. From $\Phi = \hat{\pi}_*(\nu^*\lambda)$ the reader may easily verify that $\mu$ is $\Phi$-compatible, and thus, we are left to verify that $\mu$ is symmetric. Similar to the $\Sp_{2n}$ case $\Tr_{\hat{S}/X} = \Tr_{\hat{S}/X} \circ \,\hat{\sigma}$ and there is a canonical lift $\bar{\sigma}$ of $\hat{\sigma}$ to $\hat{\sigma}^*(N^*)N^*(R).$ Hence, consider
\begin{equation}
\label{symmetric}
\mu(b,a) = \Tr_{\hat{S}/X}\hat{\sigma}(b \otimes \theta \hat{\sigma}^*a) = \Tr_{\hat{S}/X}(\hat{\sigma}^*b \otimes \bar{\sigma}^*\theta a) = \Tr_{\hat{S}/X}(a \otimes \bar{\sigma}^*\theta\hat{\sigma}^*b).
\end{equation}
From (\ref{symmetric}), $\mu$ is symmetric if and only if $\bar{\sigma}^* \theta = \theta.$ Contrary to the analogous condition in the $\Sp_{2n}$ case this condition is not immediate. However, there is a necessary and sufficient condition for $\bar{\sigma}^*\theta = \theta,$ which we will establish. Let $C = \hat{S} / \hat{\sigma}$ with induced map $q : C \to X$ such that $\hat{\pi} = q \circ \pr$ where $\pr : \hat{S} \to C$ denotes the natural projection map. Notice that since $\hat{\sigma}$ has no fixed points $\pr: S \to C$ is an unramified double cover, and thus, $\pr^* : \Jac(C) \to \Jac(\hat{S})$ is not injective. Now, $N\sigma^*(N) \cong \pr^*\Nm_{\hat{S}/C}(N)$, and since $\pr$ is unramified, $\mathcal{O}_{\hat{S}}(R) \cong \pr^*\mathcal{O}_{C}(R_q)$ where $R_q$ is the ramification divisor of $q$. It follows that we may view $\theta$ as an isomorphism 
\begin{equation}
\label{altiso}
\theta : \pr^*\Nm_{\hat{S}/C}(N) \to \pr^*\mathcal{O}_C(R_q).
\end{equation}
Now, note that although $\pr^*$ is not injective, (\ref{altiso}) descends to an isomorphism $\Nm_{\hat{S}/C}(N) \cong \mathcal{O}_C(R_q)$ if and only if $\bar{\sigma}^*\theta = \theta$. Therefore, $(E, \Phi, \mu)$ defines an $\mathrm{O}_{2n}$-Higgs bundle if and only if $\theta$ descends to an isomorphism $\Nm_{\hat{S}/C}(N) \cong \mathcal{O}_C(R_q).$ Suppose that $(E, \Phi, \mu)$ defines an $\mathrm{O}_{2n}$-Higgs bundle. It is not immediately clear that $E = \hat{\pi}_*(N)$ has trivial determinant. However, this turns out to be the case, which implies that, along with a choice of orientation, $(E, \Phi, \mu)$ defines a $\SO_{2n}$-Higgs bundle, which establishes the desired correspondence. Although there is a choice of two orientations, one preserves the Pfaffian and the other orientation reverses the Pfaffian, and thus, we choose the orientation preserving the Pfaffian. Hence, we are left to show that $\det(E) \cong \mathcal{O}_X.$ 

\subsection{$\det(E) \cong \mathcal{O}_X$}
\label{det E is trivial}

By the determinant of the pushforward formula, there is a canonical isomorphism 
\begin{equation}
\label{Eformula}
\det(E) \cong \Nm_{\hat{S}/X}(N) \otimes \det (\hat{\pi}_* \mathcal{O}_{\hat{S}}).
\end{equation}
 Since $\Nm_{\hat{S}/X} = \Nm_{C/X} \circ \Nm_{{\hat{S}/C}},$ and $\Nm_{\hat{S}/C}(N) \cong \mathcal{O}_C(R_q)$, it follows that \[\Nm_{\hat{S}/X}(N) \cong \Nm_{C/X}(\mathcal{O}_C(R_q)).\] Again applying the determinant of the pushforward formula we obtain \[\Nm_{C/X}(\mathcal{O}_C(R_q)) \cong \det(q_*\mathcal{O}_C(R_q)) \otimes (\det q_*\mathcal{O}_C)^*.\] Then, by relative duality $\det(q_*\mathcal{O}_C(R_q)) \cong (\det q_*\mathcal{O}_C)^*$, and from (\ref{Eformula}) we obtain
\begin{equation}
\label{detE}
\det(E) \cong \det( \hat{\pi}_* \mathcal{O}_{\hat{S}}) \otimes (\det q_*\mathcal{O}_C)^{-2}.
\end{equation}
From $(\ref{detE})$ it suffices to compute $\det( \hat{\pi}_* \mathcal{O}_{\hat{S}})$ and $\det q_*\mathcal{O}_C$. For $\det q^* \mathcal{O}_C$ a modification of the proof to Proposition \ref{structuresheaf} shows that $q_* \mathcal{O}_C \cong \mathcal{O}_C \oplus L^{-2} \oplus \cdots \oplus L^{-(2n-2)},$ and thus,
\begin{equation}
\label{firstiso}
\det q_* \mathcal{O}_C \cong L^{n-n^2}
\end{equation}
To compute $\det(\hat{\pi}_* \mathcal{O}_{\hat{S}})$ consider the following short exact sequence of sheaves
\begin{equation}
\label{normsheaf}
0 \to \mathcal{O}_S \to \nu_*\mathcal{O}_{\hat{S}} \to \bigoplus_{\sigma(p) = p} \CC_p \to 0
\end{equation} 
where $\CC_p$ denotes the skyscraper sheaf at $p.$ Applying $\pi_*$ to (\ref{normsheaf}) gives the long exact sequence
\begin{equation}
\label{les}
0 \to \pi_*\mathcal{O}_S \to \hat{\pi}_* \mathcal{O}_{\hat{S}} \to \bigoplus_{p_n(x) = 0} \CC_x \to R^1 \pi_*\mathcal{O}_S \to \cdots
\end{equation}
Note that $\pi^{-1}(y)$ is finite for each $y \in X$, and hence, $\HH^1(\pi^{-1}\{y\}, \mathcal{O}_S) = 0$. Thus, by Grauert's base change theorem, $R^1 \pi_*\mathcal{O}_S = 0$, and (\ref{les}) reduces to the short exact sequence
\begin{equation}
\label{secondsheafseq}
0 \to \pi_*\mathcal{O}_S \to \hat{\pi}_* \mathcal{O}_{\hat{S}} \to \bigoplus_{p_n(x) = 0} \CC_x \to 0.
\end{equation}
Let $j : \pi_*\mathcal{O}_S \to \hat{\pi}_* \mathcal{O}_{\hat{S}}$ denote the sheaf map in (\ref{secondsheafseq}) and suppose $x \in X$ is a zero of $p_n.$ Choose a local coordinate $z$ centred at $x$ such that $p_n(z) = z$. Since $(\pi_*\mathcal{O}_S)_x$ and $(\hat{\pi}_* \mathcal{O}_{\hat{S}})_x$ are rank $2n$ free $\mathcal{O}_x$-modules, where $\mathcal{O}_x$ is a principal ideal domain, we may choose a basis for $\mathcal{O}_x$ such that $\det(j)(z) = \operatorname{diag}(z^{e_1}, \ldots, z^{e_{2n}})$ is in smith normal form. However, from (\ref{secondsheafseq}) it follows that $j(z) = \diag(1, \ldots, 1, z)$ and $\det(j)(z) = z.$ Therefore, $\det j : \det(\pi_* \mathcal{O}_S) \to \det(\hat{\pi}_* \mathcal{O}_{\hat{S}})$ has a simple zero for each zero of $p_n \in \HH^0(X, L^n),$ and hence, $
\det(\hat{\pi}_* \mathcal{O}_{\hat{S}}) \cong L^n \otimes  \det(\pi_* \mathcal{O}_S).
$
From Proposition \ref{structuresheaf} and the foregoing discussion it follows that 
\begin{equation}
\label{detformulanorm}
\det(\hat{\pi}_* \mathcal{O}_{\hat{S}}) \cong L^{2n-2n^2}.
\end{equation}
Combining (\ref{detE}), (\ref{firstiso}), and (\ref{detformulanorm}) it follows that $\det(E) \cong \mathcal{O}_X,$ which completes the correspondence. That is, isomorphism classes of $\SO_{2n}$-Higgs bundles with generic spectral curve $S$ are in one-to-one correspondence with holomorphic line bundles $N \in \Pic(\hat{S})$ such that $\Nm_{\hat{S}/C}(N) \cong \mathcal{O}_C(R_q).$ Of course, choosing a bundle $M \in \Pic(\hat{S})$ such that $\Nm_{\hat{S}/C}(M) = \mathcal{O}_C(R_q)$ gives rise to an isomorphism $N \cong U \otimes M$ where $\Nm_{\hat{S}/C}(U) \cong \mathcal{O}_C.$ Thus, $\SO_{2n}$-Higgs bundles are in one-to-one correspondence with a torsor of $\ker(\Nm_{\hat{S}/C})$.

\bigskip

Since the degree 2 map $\pr : \hat{S} \to C$ is unramified, $\ker(\Nm_{\hat{S}/C})$ is not connected. Each connected component is a torsor of $\Prym(\hat{S}, C)$. The number of connected components of $\ker(\Nm_{\hat{S}/C})$ is equal to the index of $\pr_*\pi_1(\Jac(\hat{S}))$ in $\pi_1(\Jac(C)).$ By identifying the Jacobian varieties with their associated Albanese varieties there are canonical isomorphisms $\pi_1(\Jac(\hat{S})) \cong \HH_1(\hat{S}, \ZZ)$ and $\pi_1(\Jac(C)) \cong \HH_1(C, \ZZ)$. Now, the reader may verify from Hurewicz theorem that $(\HH^1(C, \ZZ) : \pr_*\HH^1(\hat{S}, \ZZ)) = (\pi_1(C): \pr_* \pi_1(\hat{S}))$. Finally, the covering map $\pr : \hat{S} \to C$ is determined by the kernel of a homomorphism $\pi_1(C) \to \ZZ_2$, and it follows that $(\HH^1(C, \ZZ) : \pr_*\HH^1(\hat{S}, \ZZ)) = 2$. Therefore, $\ker(\Nm_{\hat{S}/C})$ is comprised of two connected components, namely, $\Prym(\hat{S}, C)$ and a torsor of $\Prym(\hat{S}, C).$ In summary, isomorphism classes of $\SO_{2n}$-Higgs bundles with generic spectral curve $S$ consists of two connected components each of which are torsors of the Prym variety $\Prym(\hat{S}, C),$ which establishes the last result in Theorem \ref{main theorem}. Note, since $\pr : \hat{S} \to C$ is a unramified double cover the Prym variety is self-dual, which is exactly Langlands duality in this setting.

\bibliographystyle{plain}
\bibliography{sn-bibliography}

\begin{thebibliography}{10}

\bibitem{MR2448083}
T.~Abe.
\newblock Strange duality for parabolic symplectic bundles on a pointed
  projective line.
\newblock {\em Int. Math. Res. Not. IMRN}, pages Art. ID rnn121, 47, 2008.

\bibitem{MR3049307}
T.~Abe.
\newblock Moduli of oriented orthogonal sheaves on a nodal curve.
\newblock {\em Kyoto J. Math.}, 53(1):55--90, 2013.

\bibitem{MR4027558}
D.~Baraglia, M.~Kamgarpour, and R.~Varma.
\newblock Complete integrability of the parahoric {H}itchin system.
\newblock {\em Int. Math. Res. Not. IMRN}, (21):6499--6528, 2019.

\bibitem{MR0749574}
W.~Barth, C.~Peters, and A.~Van~de Ven.
\newblock {\em Compact complex surfaces}, volume~4 of {\em Ergebnisse der
  Mathematik und ihrer Grenzgebiete (3) [Results in Mathematics and Related
  Areas (3)]}.
\newblock Springer-Verlag, Berlin, 1984.

\bibitem{MR998478}
A.~Beauville, M.~S. Narasimhan, and S.~Ramanan.
\newblock Spectral curves and the generalised theta divisor.
\newblock {\em J. Reine Angew. Math.}, 398:169--179, 1989.

\bibitem{MR2062673}
C.~Birkenhake and H.~Lange.
\newblock {\em Complex abelian varieties}, volume 302 of {\em Grundlehren der
  mathematischen Wissenschaften [Fundamental Principles of Mathematical
  Sciences]}.
\newblock Springer-Verlag, Berlin, second edition, 2004.

\bibitem{boalch2012}
P.~P. Boalch.
\newblock Hyperkahler manifolds and nonabelian hodge theory of (irregular)
  curves, 2012.

\bibitem{chen2018}
T.~H. Chen and B.~C. Ng\^o.
\newblock On the hitchin fibration for algebraic surfaces, 2018.

\bibitem{MR4118645}
T.~H. Chen and B.~C. Ng\^o.
\newblock On the {H}itchin morphism for higher-dimensional varieties.
\newblock {\em Duke Math. J.}, 169(10):1971--2004, 2020.

\bibitem{MR0965220}
K.~Corlette.
\newblock Flat {$G$}-bundles with canonical metrics.
\newblock {\em J. Differential Geom.}, 28(3):361--382, 1988.

\bibitem{MR1903115}
R.~Y. Donagi and D.~Gaitsgory.
\newblock The gerbe of {H}iggs bundles.
\newblock {\em Transform. Groups}, 7(2):109--153, 2002.

\bibitem{MR1215934}
A.~Fr\"{o}hlich and M.~J. Taylor.
\newblock {\em Algebraic number theory}, volume~27 of {\em Cambridge Studies in
  Advanced Mathematics}.
\newblock Cambridge University Press, Cambridge, 1993.

\bibitem{MR0887284}
N.~J. Hitchin.
\newblock The self-duality equations on a {R}iemann surface.
\newblock {\em Proc. London Math. Soc. (3)}, 55(1):59--126, 1987.

\bibitem{MR0885778}
N.~J. Hitchin.
\newblock Stable bundles and integrable systems.
\newblock {\em Duke Math. J.}, 54(1):91--114, 1987.

\bibitem{MR2354922}
N.~J. Hitchin.
\newblock Langlands duality and {$G_2$} spectral curves.
\newblock {\em Q. J. Math.}, 58(3):319--344, 2007.

\bibitem{MR3618052}
N.~J. Hitchin.
\newblock Higgs bundles and characteristic classes.
\newblock In {\em Arbeitstagung {B}onn 2013}, volume 319 of {\em Progr. Math.},
  pages 247--264. Birkh\"{a}user/Springer, Cham, 2016.

\bibitem{MR4230392}
S.~Mukhopadhyay and R.~Wentworth.
\newblock Spectral data for spin {H}iggs bundles.
\newblock {\em Int. Math. Res. Not. IMRN}, (6):4211--4230, 2021.

\bibitem{MR0282985}
D.~Mumford.
\newblock {\em Abelian varieties}, volume~5 of {\em Tata Institute of
  Fundamental Research Studies in Mathematics}.
\newblock Published for the Tata Institute of Fundamental Research, Bombay by
  Oxford University Press, London, 1970.

\bibitem{MR1085642}
N.~Nitsure.
\newblock Moduli space of semistable pairs on a curve.
\newblock {\em Proc. London Math. Soc. (3)}, 62(2):275--300, 1991.

\bibitem{MR1353317}
P.~Scheinost and M.~Schottenloher.
\newblock Metaplectic quantization of the moduli spaces of flat and parabolic
  bundles.
\newblock {\em J. Reine Angew. Math.}, 466:145--219, 1995.

\bibitem{MR1040197}
C.~T. Simpson.
\newblock Harmonic bundles on noncompact curves.
\newblock {\em J. Amer. Math. Soc.}, 3(3):713--770, 1990.

\bibitem{MR1179076}
C.~T. Simpson.
\newblock Higgs bundles and local systems.
\newblock {\em Inst. Hautes \'{E}tudes Sci. Publ. Math.}, (75):5--95, 1992.

\bibitem{MR1307297}
C.~T. Simpson.
\newblock Moduli of representations of the fundamental group of a smooth
  projective variety. {I}.
\newblock {\em Inst. Hautes \'{E}tudes Sci. Publ. Math.}, (79):47--129, 1994.

\bibitem{MR1320603}
C.~T. Simpson.
\newblock Moduli of representations of the fundamental group of a smooth
  projective variety. {II}.
\newblock {\em Inst. Hautes \'{E}tudes Sci. Publ. Math.}, (80):5--79, 1994.

\bibitem{MR4686658}
Lei Song and Hao Sun.
\newblock On the image of {H}itchin morphism for algebraic surfaces: the case
  {${\rm GL}_n$}.
\newblock {\em Int. Math. Res. Not. IMRN}, (1):492--514, 2024.

\end{thebibliography}

\end{document}